\documentclass[10pt,compress]{article}

\usepackage{latexsym,amsfonts}
\usepackage{amssymb,amsthm,upref,amscd}
\usepackage[T1]{fontenc}
\usepackage{times}
\usepackage{amscd}
\usepackage{enumerate}
\usepackage{mathrsfs}
\usepackage{amsmath}
\usepackage{color}

\usepackage{indentfirst}
\usepackage{geometry}
\PassOptionsToPackage{reqno}{amsmath}
\usepackage{mathbbol}
\DeclareSymbolFontAlphabet{\amsmathbb}{AMSb}
\newcommand{\blue}[1]{\textcolor{blue}{#1}}
\newcommand{\dd}{\;{\rm d}}
\newcommand{\ee}{{\rm e}}
\newcommand{\ff}{\varphi}
\newcommand{\rr}{\mathbb{R}}
\newcommand{\lt}{{L^2}}

\newcommand{\scal}[1]{\left\langle #1 \right\rangle}
\newcommand{\rt}{{\mathbb{R}^2}}
\newcommand{\abso}[1]{\left|   #1   \right|}
\newcommand{\norm}[1]{\left\|   #1   \right\|}
\newcommand{\paar}[1]{\left(   #1   \right)}
\newcommand{\sett}[1]{\left\{   #1   \right\}}
\newcommand{\nl}{\mathscr{L}}
\newcommand{\nh}{\mathcal{H}}
\newcommand{\ii}{{\rm i}}

\newcommand{\fc}{\mathfrak{C}}

\newcommand{\fm}{\mathfrak{M}}

\newcommand{\ft}{\mathfrak{T}}

\newtheorem{thm}{Theorem}[section]
\newtheorem{prop}{Proposition}[section]
\newtheorem{defi}{Definition}[section]
\newtheorem{lem}{Lemma}[section]
\newtheorem{cor}{Corollary}[section]

\newtheorem{rem}{Remark}[section]
\theoremstyle{notation}

\newcommand{\R}{\mathbb{R}}

\numberwithin{equation}{section}

\newcommand{\Z}{\mathbb{Z}}

\newcommand{\E}{J}

\newcommand{\al}{\alpha}

\newcommand{\om}{\omega}

\newcommand{\eps}{\epsilon}

\newcommand{\bW}{{\mathbb W}}

\newcommand{\wto}{\rightharpoonup}

\makeatletter \@addtoreset{equation}{section} \makeatother

\usepackage[colorlinks=true,urlcolor=blue,
citecolor=blue,linkcolor=blue,linktocpage,pdfpagelabels,
bookmarksnumbered,bookmarksopen,pagebackref]{hyperref}
\newcounter{const}
\usepackage[colorinlistoftodos]{todonotes}
\makeatletter
\providecommand\@dotsep{5}
\def\listtodoname{List of Todos}
\def\listoftodos{\@starttoc{tdo}\listtodoname}
\makeatother
\allowdisplaybreaks

\usepackage{fancyhdr}
\author{Amin Esfahani\footnote{Department of Mathematics, Nazarbayev University, Astana 010000, Kazakhstan.   E-mail: amin.esfahani@nu.edu.kz}, Hichem Hajaiej\footnote{Department of Mathematics, Cal State LA,
Los Angeles CA 90032, USA.
Email: hhajaie@calstatela.edu}, Alessio Pomponio\footnote{Dipartimento di Meccanica, Matematica e Management,	Politecnico di Bari,
via Orabona 4,  70125  Bari, Italy. E-mail:
alessio.pomponio@poliba.it}   }

\title{New insights into the solutions of a class of anisotropic nonlinear Schr\"odinger equations on the plane   
 }

\date{} 
\begin{document}

\maketitle

 \begin{abstract}
		In this paper, we study the following anisotropic nonlinear Schr\"odinger equation on the plane,
	\[\begin{cases}
	    \ii\partial_t \Phi+\partial_{xx} \Phi -D_y^{2s}  \Phi +|\Phi|^{p-2}\Phi=0,&\quad    (t,x,y)\in\mathbb{R} \times \mathbb{R}^2,\\
     \Phi(x,y,0)=\Phi_0(x,y),&\quad  (x,y)\in\rt,
	\end{cases}\]
		 		where  $D_y^{2s}=\left(-\partial_{yy}\right)^s$ denotes the fractional Laplacian with $0<s<1$ and $2<p<\frac{2(1+s)}{1-s}$. We first study the existence of normalized solutions to this equation in the subcritical, critical, and supercritical cases.  To this aim, regularity results and a Pohozaev type identity are necessary. Then, we  determine the conditions under which the solutions blow up. Furthermore, we demonstrate the existence of boosted traveling waves when $s\geq1/2$ and their decay at infinity. Additionally, for the delicate case $s=1/2$, we provide a non-existence result of boosted traveling waves and we establish that there is no scattering for small data. Finally, we also study normalized boosted travelling waves in the mass subcritical case.  Due to the nature of the equation, we do not impose any radial symmetry on the initial data or on the solutions.  \\ \\

     
   \end{abstract}

\tableofcontents
	\section{Introduction}

In recent decades, there has been significant interest among mathematicians in the study of nonlinear fractional Schr\"{o}dinger equations, primarily due to their applications in nonlinear optics. Notably, Laskin introduced the concept of using the path integral over Lévy trajectories instead of Brownian trajectories to deduce the fractional Schr\"{o}dinger equation \cite{laskin1,laskin2}. In the context of fractional quantum mechanics, the Hamiltonian can be given as follows:
\[H_s(\mathbf{p},\mathbf{r})=|\mathbf{p}_1|^2+|\mathbf{p}_2|^{2s}.\]
Following the argument presented in \cite{laskin1} (see also \cite{copac}), we arrive at the anisotropic nonlinear Schr\"{o}dinger equations with local nonlinearities:
\begin{equation} \label{eq0}
	\begin{cases}
	    \ii\partial_t \Phi+\partial_{xx} \Phi -D_y^{2s}  \Phi +|\Phi|^{p-2}\Phi=0,&\quad \text{for }  (t,x,y)\in\mathbb{R} \times \mathbb{R}^2,\\
     \Phi(x,y,0)=\Phi_0(x,y),&\quad \text{for }  (x,y)\in\rt,
	\end{cases}
\end{equation}
where $D_y^{2s}=\left(-\partial_{yy}\right)^s$ denotes the fractional Laplacian with $0<s<1$ and $p>2$. Equation \eqref{eq0} models important problems in physics and engineering, see \cite{YIK} and references therein for more details.

The fractional Laplacian $\left(-\partial_{yy}\right)^s$ is defined by its Fourier transform as $\mathcal{F}\left(D_y^{2s}u\right)(\xi)=|\xi_2|^{2s}\mathcal{F}(u)(\xi)$, where $\xi=(\xi_1,\xi_2)$ and $\mathcal{F}$ represents the Fourier transform. The operator $\mathcal{L}=-\partial_{xx} +D_y^{2s} $ arises in the investigation of toy models \cite{BCCI1, BCCI2} for parabolic equations exhibiting local diffusion solely in the $x$-direction, while non-local diffusion, capturing long-range interactions, occurs in the $y$-direction. Furthermore, $\mathcal{L}$ appears in the study of the two-dimensional generalized Benjamin-Ono-Zakharov-Kuznetzov equation, serving as a model for electromagnetic phenomena in thin nano-conductors on a dielectric substrate (see \cite{epb}).
Equation \eqref{eq0} was initially explored in \cite{Xu}, where the scattering of solutions to \eqref{eq0} defined on $\R_x\times \mathbb{T}_y$ with $s=\frac{1}{2}$ and $p=4$ was discussed. Subsequently, in \cite{YIK}, the local and global well-posedness of solutions to \eqref{eq0}, along with the existence and orbital stability of solitary waves for $1/2<s<1$ and $2<p<6$, were investigated.

Equation \eqref{eq0} enjoys (formally) the following conservation laws:
\begin{align*}
	\text{Mass:}\quad &M(\Phi)=\int_{\R^2}|\Phi(t)|^2\dd x\dd y,\\
	\text{Energy:}\quad &E(\Phi)=\frac12\int_{\R^2}\paar{|\Phi_x(t)|^2+|D_y^s\Phi(t)|^2}\dd x\dd y-\frac1p\int_\rt|\Phi(t)|^p\dd x\dd y,\\
	\text{Angular  Momentum:}\quad & P(\Phi)=\ii\scal{D_y^{\frac12}\Phi(t),\nh_yD_y^{\frac12} \Phi(t)}+
	\ii\scal{ \Phi(t),  \Phi_x(t)},
\end{align*}
 where $\nh_y$ represents the Hilbert transform in the $y$-variable.

The study of nonlinear fractional Schr\"{o}dinger equations, particularly with anisotropy, provides valuable insights and has significant applications in many domains, including nonlinear optics, ecology, biology, and fractional quantum mechanics.
Seeking solitary wave solutions (solitons) is crucial to understand the behavior of particles  in quantum mechanics. A soliton refers to a particle-like solution to the nonlinear equation that retains its shape and speed while propagating. In this context, soliton solutions take the form:
\[\Phi(t,x,y)=e^{\ii\al t} u(x,y),\]
where $\al\in \mathbb{R}$ represents the frequency of the wave. Equation \eqref{eq0} then reduces to the following equation for $u(x,y)$:
\begin{align} \label{equ-0}
	-\partial_{xx} u + D_y^{2s} u + \al u=|u|^{p-2}u \quad \mbox{in} \ \R^2.
\end{align}

Addressing the existence of normalized solutions provides valuable information in the understanding of the model under-study and its realization.
While the linear Schr\"{o}dinger equation possesses the universality property allowing solutions to be freely normalized, the situation is not as straightforward for Schr\"{o}dinger equations with polynomial nonlinearities.  

In this work, we first aim to find normalized solutions of equation \eqref{equ-0}, which are subject to the $L^2$-norm constraint:
\begin{align} \label{mass}
	\int_{\R^2} |u|^2 \dd x\dd y =c>0.
\end{align}

Normalized solutions of nonlinear Schr\"{o}dinger equations have garnered substantial attention within the mathematical community.
The first breakthrough paper providing a general approach to establish the existence of normalized solutions is due to Lions \cite{lions85}. In the celebrated concentration-compactness method, the author presented a heuristic method to establish the existence of normalized solutions in the subcritical case.  Hundreds of papers used this approach to study the existence of normalized solutions for particular nonlinearities. In \cite{HS1}, the authors were able to develop Lions' method and included a large class of nonlinearities.  The critical and the supercritical cases remained unsolved until the brilliant approach of   Bellazzini and Jeanjean \cite{Bellazzini-Jeanjean}   who suggested a new constraint for this challenging case. See also \cite{jeanjean}. More recently,   Hajaiej and   Song developed novel and self-contained approaches that not only provide the key ingredients to study the existence/non-existence of normalized solutions but also addresses their multiplicity or uniqueness \cite{HA1, hajsong, Hajaiej-Song-unique}. However, all the existing results do not seem to directly apply to the anisotropic setting where some fundamental properties of this operator have not been addressed yet. To the best of our knowledge, the maximum principle, the Pohozaev identity as well as the regularity, and decay properties of the solutions to \eqref{equ-0} seem to all be missing in the literature.  We will shed some lights on some of these fundamental properties of our anisotropic operator. This presented some difficulties and required new ideas contrarly to the study of normalized solutions in the fractional setting that didn't require novelties  but rather added considerable technical difficulties except for the regularity. The reader can refer to \cite{HH1, HH2, HH3}, where the existence/non-existence,  orbital stability of standing waves under optimal nonlinearities have been established.

In the  first part of the  present paper, we address the solutions to \eqref{equ-0} under the $L^2$-norm constraint \eqref{mass}, $2<p<\frac{2(s+1)}{1-s}$, and the parameter $\alpha \in \R$ as an unknown Lagrange multiplier. Notably, \cite{GHS} demonstrated the existence of a positive ground state for equation \eqref{equ-0} when $p<\frac{2(1+s)}{1-s}$, by exploring the axial symmetry. In this context, a ground state   is defined as the solution to \eqref{equ-0} with the lowest energy among all solutions. Additionally, \cite{GHS} revealed that these waves exhibit spectral stability when $p\leq\frac{2(1+s)}{1-s}$ and spectral instability within the complementary range $\frac{2(3s+1)}{1+s}<p<\frac{2(1+s)}{1-s}$. Uniqueness results were also presented in \cite{GHS} under non-degeneracy abstract conditions. As established in \cite{hajsong} the study of normalized solutions is much harder than the one of ground states. %In a more recent work \cite{H-recent}, addressed the existence and non-existence of solutions of a system.

 Here, solutions to \eqref{equ-0}-\eqref{mass} correspond to critical points of the energy functional $E$ restricted on $S_c$, where
\[
S_c:=\left\{u \in H : \|u\|^2_2  =c>0\right\} 
\]
and $H$ is defined via the norm
\[
\|u\|_H^2=\|u\|_{2}^2+
\|u_x\|_{2}^2+
\|D_y^{s}u\|_{2}^2.
\]
Note that $H=H_x^1L_y^2\cap L^2_xH^s_y$.  In this situation the parameter $\al$ arises as a Lagrange multiplier, it does depend on the
solution and is not a priori given.

We will show that the energy $E$ is bounded from below on $S_c$ if $2<p<\frac{2(3s+1)}{1+s}$ (called subcritical case). This is an easy corollary of an anisotropic Gagliardo-Nirenberg inequality associated with \eqref{equ-0} (see Lemma \ref{gn-lemma} and Theorem \ref{thm1}). On the other hand, the boundedness of $E$ on $S_c$ fails when $p>\frac{2(3s+1)}{1+s}$ (supercritical case). Indeed, for $u \in S_c$ and $t>0$, we introduce a scaling of $u$ as	 
\begin{equation}\label{scaling-0}
		u_t(x, y):=t^{\frac{s+1}{2}} u(t^s x, t y), \quad (x, y) \in \R^2.
\end{equation}
Thus, we get that $u_t \in S_c$ and
	\begin{align} \label{scaling}
		E(u_t)=\frac{t^{2s}}{2} \|\partial_x u\|_2^2  + \frac{t^{2s}}{2} \|D_y^{s} u\|_2^2  -\frac{t^{\frac{(s+1)(p-2)}{2}}}{p} \|u\|_p^p .
	\end{align}
As we let $t\to+\infty$, it becomes evident that the energy $E$ constrained on $S_c$ becomes unbounded from below for any $c > 0$. Consequently, a global minimizing approach fails to be effective in the supercritical case.  Inspired by the concept of the minimizing method on the Pohozaev manifold, our objective is to create a submanifold within $S_c$ where $E(u)$ is both bounded from below and coercive. Subsequently, we aim to identify minimizers of $E(u)$ on this specific submanifold. To achieve this, we present the following minimizing problem:
\[
\inf_{u \in P_c} E(u),
\]
 where
\begin{equation}\label{Pc}
    P_c:=\{u \in S_c : Q(u)=0\},
\end{equation}
and 
\begin{equation}\label{Q}
Q(u):= s \paar{\norm{u_x}_2^2+\norm{D_y^su}_2^2} -\frac {(s+1)(p-2)}{2p} \|u\|_{p}^p.
\end{equation}
The submanifold $P_c$ is known as the Pohozaev-Nehari manifold. We demonstrate that it constitutes a smooth manifold of codimension 2 in $H$ for any $c>0$. Additionally, we establish the coercivity of $E$ on $P_c$.

Shifting focus, our meticulous observations reveal that the energy functional $E\vert_{S_c}$ exhibits a mountain pass geometry. This discovery propels us towards seeking a critical point utilizing the mountain pass theorem. This strategy heavily relies on an enhanced iteration of the min-max principle, as originally introduced by Ghoussoub \cite{Gh}. 

It is noteworthy that the Pohozaev identity $Q(u) = 0$ was already established for the ground states of \eqref{equ-0} in \cite{GHS}, employing an idea related to the maximum principle derived from \cite{FW}. However, our approach does not assume that the solutions to \eqref{equ-0} are non-negative real-valued. This constitutes a major restriction and requires novel ideas to address. The distinctive contribution of our work lies in proving the Pohozaev identity for any solution of \eqref{equ-0} while relaxing the assumptions mentioned earlier (see Theorem \ref{ph}).
To demonstrate the validity of this identity, we establish appropriate regularity and decay properties of the solutions, which are the main ingredients. Our method is inspired by \cite{FW}, where Felmer and Wang constructed suitable super and sub-solutions to obtain a version of the Hopf Lemma and a strong maximum principle, leading to suitable decay of positive classical solutions of \eqref{equ-0}. However, in our case, we do not restrict ourselves to these solutions and establish the essential decay and regularity of any solution $u\in H$ of \eqref{equ-0}, improving the results of \cite{FW}. To achieve this, we employ the methods developed by   Bona and Li in \cite{bonali} (see also \cite{capni,epb}) concerning the asymptotic properties of solitary-wave solutions of the one-dimensional dispersive equation. Their results heavily rely on the kernel of the associated linear equation.
It is worth noting that an important feature of the results of \cite{bonali} is that the solutions are not necessarily positive. This aspect aligns with our direction of proof for $Q(u)=0$. Nevertheless, a principal challenge in achieving this goal is that the kernel of the linear equation associated with \eqref{equ-0} cannot be explicitly presented using Cauchy's Theorem or Residue Theorem. The anisotropic nature of the equation adds complexity to the analysis, as evidenced by Lemma \ref{lem1} and Theorem \ref{decay-regul-lemma}.

In light of the aforementioned results, we proceed to demonstrate the existence of energy minimizers (normalized solutions) of \eqref{equ-0} on the sphere $S_c$. In Sections \ref{subcritical-section} and \ref{critical-section}, we first investigate the subcritical case $p<\frac{2(3s+1)}{s+1}$ and the critical case $p=\frac{2(3s+1)}{s+1}$, respectively. In the subcritical case, the existence of normalized solutions follows directly from the boundedness of $E$ from below on $S_c$   and a standard minimization argument (see Theorem \ref{thm1}). Moreover the set of minimizers of $E$ on $S_c$ is orbitally stable, as shown in Theorem \ref{thm2}. 
In the critical case, we demonstrate that no normalized solution exists when $c$ is small in Theorem \ref{thm1-mass-cr}.
In the supercritical case, in Section \ref{supercritical-section} we leverage the concepts introduced in \cite{jeanjeanlu2} and consider more general nonlinearities to establish the existence of normalized solutions of \eqref{equ-0}. Furthermore, by employing homotopy and a version of Ghossoub's min-max principle, we succeed in proving the existence of ground states in Theorem \ref{groun-thm}.

 In the second part of this paper with Section \ref{blowup-section},  we present an independent result by identifying criteria for the existence of blow-up solutions in the Cauchy problem associated with \eqref{eq0}. In \cite{BHL}, implementing a localized virial law and by means of the representation Balakrishnan’s formula, the authors demonstrated the blow-up of  solutions of the fractional nonlinear Schr\"odinger  equation in the critical or supercritical case in the radial setting. 
In the estimates, therein, the radial symmetry, together with the associated uniform decay at infinity of radial functions (due to a generalized Strauss Lemma for fractional Sobolev space), is a crucial assumption. 
Here, we adapt some ideas of \cite{BHL} but, unlike the fractional nonlinear Schr\"odinger equation, our anisotropic problem  \eqref{eq0} lacks natural  radial symmetry, necessitating the removal of this condition and modification of the argument to apply it to our case. 
In particular we have to require some additional regularity and integrability of the local solution of \eqref{eq0} obtaining some criteria for the blow-up phenomena in Theorem \ref{bl-c-thmm}. Consequently, in Theorem \ref{unstable} are able to demonstrate the instability of standing waves by exploiting the mechanism of blow-up.

 We conclude with  Section \ref{boosted-section}, where we explore the existence of boosted traveling waves represented by the form $\ee^{\ii\al t} u(x,y-\omega t)$, with $\al,\omega\in\rr$. These waves exhibit a non-zero speed $\omega$. Hong and Sire \cite{hong-sire} achieved similar solutions for the one-dimensional fractional nonlinear Schr\"odinger equation with $s\in(1/2,1)$ and cubic nonlinearity. They addressed the fractional nonlinear Schr\"odinger equation's lack of Galilean invariance by introducing an ansatz function derived from pseudo-Galilean transformations. These transformations ensure the fractional nonlinear Schr\"odinger  equation remains invariant, albeit with a controllable error term. 
 Here, drawing inspiration from \cite{bglv}, we initiate our approach with a Weinstein-type minimization problem. Through this, we establish the existence of a nontrivial boosted traveling wave, also known as a ground state, when $s\geq1/2$ (and suitable $\omega$). Addressing the issue of compactness within the minimizing sequence of the aforementioned minimization problem, we introduce a generalization of Lieb’s compactness lemma.   
 Subsequently, utilizing the Steiner symmetrization with respect to $x$ in the Fourier space, as outlined in \cite{BHL}, we establish the symmetry properties of these boosted ground states. We also study the decay of such solutions. See Theorems \ref{exis-thm-bt} and \ref{decay-thm-re}. 
However, in the case where $s=1/2$, the scenario becomes more intricate due to the competing effects of $\ii\partial_y$ and $D_y^1$. This intricacy leads us, in Theorem \ref{nonexi} by means of a Pohozaev type identity, to the conclusion that no nontrivial boosted travelling wave exists if $|\omega|>1$. Later on, we demonstrate the absence of small data scattering for \eqref{eq0} within the energy space $H$, for $s=1/2$, in Theorem \ref{nonscat}. At the end, we also prove the existence of normalized boosted travelling waves in the mass subcritical case, in Theorem \ref{boost-norm}.

	\section{Regularity and Pohozaev type identity}\label{se:reg}
	
	In this section,  after some preliminary results, we show regularity properties for solutions of \eqref{equ-0} which will provide a Pohozaev type identity.  
 
 First of all, we shall present the following anisotropic Gagliardo-Nirenberg  inequality in $H$ and some properties of $H$.
	
	\begin{lem}  {\cite[Theorem 1.1]{E}}\label{gn-lemma}
		Let  $s\in (0,1)$ and $2 \leq q < q_s=\frac{2(s+1)}{1-s}$, then there exists a constant $C_{q,s}>0$ such that
		\begin{align} \label{gn}
			\int_{\R^2}|u|^q \,\dd x\dd y \leq C_{q, s} \left(\int_{\R^2}|u|^2  \dd x\dd y \right)^{\frac q2 -\frac{(q-2)(s+1)}{4s}}\left(\int_{\R^2} |\partial_x u|^2  \dd x\dd y \right)^{\frac{q-2}{4}}\left(\int_{\R^2} |D_y^{s} u|^2  \dd x\dd y \right)^{\frac{q-2}{4s}}.
		\end{align}
 If $s>1/2$, then $q=q_s$ is allowed. The sharp constant $C_{q,s}$ in \eqref{gn} is presented by
% \begin{equation}\label{best-con}
%{\color{red}	 C_{q,s}^{-1}=
% (2qs)^{-1}
% ((q-2)(1+s))^{\frac{(q-2)(1+s)}{4s}}
% \paar{2(1+s)-q(1-s)}^\frac{4s-(q-2)(1+s)}{4s}
% \|\ff\|_{L^2}^{q-2},}
%\end{equation}
 \begin{equation}\label{best-con} 
 C_{q,s}^{-1}=
 (2qs)^{-1}
 (q-2)^{\frac{(q-2)(1+s)}{4s}}s^{\frac{q-2}{4}}
 \paar{2(1+s)-q(1-s)}^\frac{4s-(q-2)(1+s)}{4s}
 \|\ff\|_{L^2}^{q-2},
\end{equation}
 where $\ff$ is a ground state of \eqref{equ-0} with $p=q$ and $\alpha =1$.
	\end{lem}
	\begin{cor}\label{cor-gn}
By employing the same argument as in \cite{E}, one can demonstrate the existence of a constant  $C_H>0$ such that
		\begin{align} \label{gn-hH}
			\int_{\R^2}|u|^q \,\dd x\dd y \leq C_H
    \left(\int_{\R^2}|u|^2  \dd x\dd y \right)^{\frac q2 -\frac{(q-2)(s+1)}{4s}}
    %\|u\|_{L^2}^{q -\frac{(q-2)(s+1)}{2s}}
    \paar{\int_\rt\left(|\partial_x u|^2+|D_y^su|^2\right)\dd x\dd y}^{\frac{(q-2)(s+1)}{4s}},
		\end{align}
  and the sharp constant $C_H$ is
  \[
  C_H^{-1}=
  (2qs)^{-1}
 ((q-2)(s+1))^{\frac{(q-2)(1+s)}{4s}} 
 \paar{2(1+s)-q(1-s)}^\frac{4s-(q-2)(1+s)}{4s}
 \|\ff\|_{L^2}^{q-2}.
  \]
	\end{cor}
	\begin{lem} \cite[Corollary 2.1]{E} \label{embedding}
			 	Let   $s\in (0,1)$. 
		The Sobolev space $H$ is continuously embedded into $L^q(\R^2)$ for any $2 \leq q  <q_s$ and it is compactly embedded into $L^q_{loc}(\R^2)$ for any $2 \leq q < q_s$. If $s>1/2$ the continuous embedded holds also for $q=q_s$. 
	\end{lem}
	
	\begin{lem} \label{concentration}
		Let  $s\in (0,1)$ and  $\{u_n\} \subset H$ be a bounded sequence in $H$ and
		$$
		\sup_{z \in \R^2} \int_{B_R(z)} |u_n|^q  \dd x\dd y=o_n(1)
		$$
		for some $R>0$, then $u_n \to 0$ in $L^q(\R^2)$ as $n \to \infty$ for any $2 <q < q_s$.
	\end{lem}
	\begin{proof}
		To prove this lemma, one can directly adapt the ideas of the proof of \cite[Lemma 1.21]{W}, then we omit its proof here.
	\end{proof}
	
		%%%%%%%%%%%%%%%%%%%
 Now we prove asymptotic  and regularity properties of solutions to \eqref{equ-0} which will be crucial to prove a Pohozaev type identity. 
  
 For the sake of simplicity, we assume that $\alpha=1$, and define the kernel
		\begin{equation*}%\label{kernel}
			K_s(x,y)=\left(\frac{1}{1+|\eta|^{2s}+|\xi|^2}\right)^\vee=C_s\int_0^{+\infty}\ee^{-t}\ee^{-\frac{x^2}{4t}}t^{-\frac{1}{2}}H_s(y,t)\;\dd t,
		\end{equation*}
		with
		\[
		H_s(y,t)=\int_{\rr}\ee^{-t|\eta|^{2s}}\ee^{i y\eta}\dd\eta.
		\]

 \begin{lem}\label{lem1}Let $s\in(0,1)$ and $k,m=0,1$. The  following properties hold.
			\begin{enumerate}[(i)]
				\item $K_s(x,y)>0$, for $(x,y)\in\rr^2$, and is an even function which is strictly decreasing in $|x|$ and $|y|$ and smooth for $x\neq0$. Moreover,  $\partial_x^m\partial_y^kK_s\in L_x^rL_y^q(\rr^2)\cap L_y^qL_x^r(\rr^2)$, for $r,q\geq1$  with 	
					\[
					s\paar{1+\frac1r}>ms+1+k-\frac1q.
					\]

	 		\item There hold that
 \begin{equation}\label{est-1}
 		|\partial_x^m\partial_y^kK_s(x,y)|\lesssim |y|^{s-1-k-2sm}|x|^m\ee^{-|x|},\qquad(x,y)\in\rr^2,
 \end{equation}
 				\begin{equation}\label{1.1-3}
 					|\partial_x^m\partial_y^kK_s(x,y)|\lesssim|x|^m|y|^{-1-2s-k-2ms}\ee^{-\frac{|x|}{4}},\qquad\mbox{if}\;|y|\geq1,
 				\end{equation}
 				and
 				\begin{equation}\label{1.1-4}
 				|\partial_x^m\partial_y^k	K_s(x,y)|\gtrsim |x|^m|y|^{-1-2s-k}\ee^{-\frac{x^2}{4}},\qquad\mbox{if}\;|y|\geq1\geq|x|.
 				\end{equation}

			\end{enumerate}
	\end{lem} 
	
 \begin{proof}
		The exponential decay properties of $H_s$ have been thoroughly investigated in \cite{FQT} (also see \cite{bg}). Specifically, it has been rigorously demonstrated that
		\[
		\lim_{|y|\to+\infty}|y|^{1+2s}H_{s}(y,1)<\infty.
		\]
		This crucial estimate, along with the scaling property
		\[
	H_s(y,t)=t^{-\frac{1}{2s}}H_s\left(t^{-\frac{1}{2s}}y,1\right)\approx\frac{t}{\paar{t^{\frac1{2s}}+|y|}^{1+2s}},
		\]
		facilitates the conclusion (see \cite{grecoian}) that
		\begin{equation}\label{1.1-0}
			\partial_y^k	H_s(y,t)\approx\min \sett{t^{-\frac{1+k}{2s}},t|y|^{-1-2s-k}},\qquad k=0,1,2,\cdots .
		\end{equation}
		Based on this approximation, we obtain  the following estimate:
		\begin{equation}\label{1.1-1}
	\abso{	\partial_x^m\partial_y^k	K_s(x,y)}
	\lesssim|y|^{-1-2s-k}|x|^m\int_{0}^{|y|^{2s}}\ee^{-t-\frac{x^2}{4t}}t^{\frac{1}{2}-m}\dd t
			+
			|x|^m\int_{|y|^{2s}}^{\infty}\ee^{-t-\frac{x^2}{4t}}t^{-\frac{1}{2}\left(1+\frac{k+1}{s}\right)-m}\dd t, 
		\end{equation}
		for $k,m=0,1$.
	To get \eqref{est-1}, we can further simplify this by applying a change of variable and utilizing the elementary inequality $t+\frac{x^2}{4t}-|x|=\frac1{4t}(2t-|x|)^2\geq0$, for all $t>0$, to derive
		\begin{equation*}%\label{1.1-2}
			|\partial_x^m\partial_y^kK_s(x,y)|\lesssim|y|^{s-1-k-2sm}|x|^m\ee^{-|x|}.
		\end{equation*}

To establish \eqref{1.1-3}, we focus on the case when $|y|\geq1$. By  the inequality
			\[
			t+\frac{x^2}{4t}\geq\frac{|x|}{4} +\frac{15}{16}t,
			\]
we have
\begin{equation}\label{primopezzo}
\begin{split}
|x|^m|y|^{-1-2s-k}\int_{0}^{|y|^{2s}}
	\ee^{-t-\frac{x^2}{4t}}t^{\frac{1}{2}-m}\dd t
&\le |x|^m|y|^{-1-2s-k}\int_{0}^{|y|^{2s}}\ee^{-\frac{15}{16}t-\frac{|x|}{4}}t^{\frac{1}{2}-m}\dd t\\&
\le |x|^m|y|^{-1-2s-k}\ee^{-\frac{|x|}{4}}\int_{0}^{\infty}\ee^{-\frac{15}{16}t}t^{\frac{1}{2}-m}\dd t\\&
\le C|x|^m|y|^{-1-2s-k}\ee^{-\frac{|x|}{4}}
\end{split}
\end{equation}
and, since $t\ge |y|^{2s}\ge 1$ and so $\ee^{-\frac{15}{16}t}\le C t^{-\frac 32}$, 
\begin{equation}\label{secondopezzo}
\begin{split}
		|x|^m\int_{|y|^{2s}}^{\infty}
		\ee^{-t-\frac{x^2}{4t}}t^{-\frac{1}{2}\left(1+\frac{1+k}{s}\right)-m}\dd t
&\le
	|x|^m\int_{|y|^{2s}}^{\infty}\ee^{-\frac{15}{16}t-\frac{|x|}{4}}t^{-\frac{1}{2}\left(1+\frac{1+k}{s}\right)-m}\dd t\\&
\le
|x|^m\ee^{-\frac{|x|}{4}}	
\int_{|y|^{2s}}^{\infty}t^{-2-\frac{1+k}{2s}-m}\dd t
\\&
\le
C|x|^m|y|^{-1-2s-k-2ms}\ee^{-\frac{|x|}{4}}.
\end{split}
\end{equation}
So by \eqref{1.1-1}, \eqref{primopezzo} and \eqref{secondopezzo}, we prove \eqref{1.1-3}.

		Finally, to demonstrate \eqref{1.1-4}, we consider $|y|\geq1\geq|x|$, and using \eqref{1.1-0} along with $\frac{x^2}{t}\leq x^2+\frac1t$, we find that
		\[
		\abso{\partial_x^m\partial_y^kK_s(x,y)}\gtrsim|x|^m|y|^{-1-2s-k}\int_0^1\ee^{-t}\ee^{-\frac{x^2}{4t}}t^{\frac12-m}\dd t\gtrsim|x|^m|y|^{-1-2s-k}\ee^{-\frac{x^2}{4}}.
		\]
		The properties of $K_s$ in (i) can be deduced from the positivity and monotonicity of $H_s$ in \cite{chenbona,FQT,frank-len}. Lastly, property (i) is obtained from \eqref{1.1-1} and the Minkowski inequality. More precisely, we can see from \eqref{1.1-0} that
\[
 		\norm{\partial_y^kH(y,t)}_{L^q_y}\lesssim t^{-\frac{1}{2s}\paar{1+k-\frac1q}},\qquad q>1,\;k=0,1.
		\]
 
As a result, \[
\begin{split}
	\norm{\partial_x^m\partial_y^kK_s }_{L_y^qL_x^r},	\norm{\partial_x^m\partial_y^kK_s }_{L_x^rL_y^q}&\lesssim
	\int_0^{+\infty}
	\ee^{-t}t^{-\frac12}\norm{\partial_x^m\ee^{-\frac{x^2}{4t}}}_{L_x^r}
\norm{\partial_y^kH_s(y,t) }_{L_y^q}\dd t\\&
	\lesssim	\int_0^{+\infty}
	\ee^{-t}t^{-\frac12}t^{\frac{1}{2r}-\frac m2}
	t^{-\frac1{2s}\paar{1+k-\frac1q}}\dd t<\infty, 
	\end{split}\]
provided that
	\[
	s\paar{1+\frac1r}>ms+1+k-\frac1q,
	\]
 so that, $\partial_x^m\partial_y^kK_s\in L_x^rL_y^q\cap L_y^qL_x^r$ for $r\geq1$, $q>1$ with $m,k=0,1$.
	 \newline\indent Finally, $q>1$ can be extended to $q=1$ by $\|K_s\|_{L^1}=\hat{K}_s(0)=1$.
	\end{proof} 
 \begin{cor}\label{decay-k}
   Let $s\in (0,1)$.  There holds that
$x\partial_{x}K_s,xK_s,yK_s,y\partial_{y}K_s\in L^q_yL_x^r$,  for any $q,r\geq1$.
 \end{cor}
 Define $H^{(s_1,s_2)}$, similar to $H$, via the norm
		\[
		\|u\|_{H^{(s_1,s_2)}}^2=\|u\|_{L^2}^2+
		\|D_x^{s_1}u\|_{L^2}^2+
		\|D_y^{s_2}u\|_{L^2}^2.
		\] 
It is not hard to see that $H^{(s_1,s_2)}\hookrightarrow L^\infty$ if   
$ \frac1{s_1}+\frac1{s_2}<2.$
  
One can see the   embedding properties of $H^{(s_1,s_2)}$  in \cite{E}.
The following Theorem can be also proved by the arguments used in \cite{epb}.

Now we have all the ingredients to prove our main result about the decay and the regularity of solutions of \eqref{equ-0}. 
		\begin{thm}
	\label{decay-regul-lemma}
		Let $s\in (0,1)$ and $2<p<\frac{2(1+s)}{1-s}$.	Any solution $u\in H $ of \eqref{equ-0} belongs to $H^{(3,2s+1)}$. Moreover $u,u_x\in C_\infty$, the space of continuous functions   vanishing at
			infinity. Furthermore, for all $(x,y)\in\rr^2$, it holds, for any   solution $u\in H$, that
			\begin{equation}\label{bounds}
				C_1 (1+|y|)^{-1-2s}\ee^{-\frac{|x|}{4}}\le u(x,y)\leq C_2(1+|y|)^{-1-2s}\ee^{-\frac{|x|}{4}},
			\end{equation}
for some $C_1,C_2\in \R\setminus\{0\}$, potentially negative.
		   Moreover $u_y\in C_\infty$, if $s>2/5$. 
	\end{thm}
	
 \begin{rem}
 	Similar to \cite[Lemma B.1]{frank-len}, the assumption $u\in H$ in Theorem \ref{decay-regul-lemma} can be replaced with $u\in L^2\cap L^p$.
 \end{rem}
		\begin{proof}

Let $u\in H $ be a solution of \eqref{equ-0}. As first step, we prove that  $u\in L^\infty$.

Since $u\in H $, it can be shown that $u$ satisfies
\begin{equation}\label{integralform}
	u(x,y)=K_s\ast g(u) ,\quad(x,y)\in\rr^2,
\end{equation}
where $g(u)=|u|^{p-2}u$, and  $u\in L^q$, with $2\leq q\leq\frac{2(1+s)}{1-s}=:q_0$, due to the embedding $H \hookrightarrow L^q$. We recall from Lemma \ref{lem1} (i) that $K_s\in L^{q}$, for $1\leq q<\frac{1+s}{1-s}$. 
 
Hence, by applying the Young inequality, it follows from \eqref{integralform} that $u\in L^{q}$ provided
\[
2\leq q<\frac{2(1+s)}{(p-1)(1-s)-4s}=:q_1.
\]
We note that $\frac{2(1+s)}{(p-1)(1-s)-4s}>\frac{2(1+s)}{1-s}$. Then, $u\in L^\infty$ if $p<\frac{1+3s}{1-s}$. Another iteration gives $u\in L^{q}$, where $$2\le q<\frac{2(1+s)}{(p^2-1)(1-s)-2p(1+s)}=:q_2.$$
Hence, $u\in L^\infty$ if $p<\frac{1+s+2\sqrt{s}}{1-s}$.
By repeating this process, we get $u\in L^{q}$, for $2\le q<q_n$, where
\[
\frac{1}{q_n}=\frac{2}{q_0}-1+\frac{p-1}{q_{n-1}}
=\frac{a_n(p)}{q_0},\quad n\geq2,
\]
%where $q_0=\paar{ \frac{2(1+s)}{1-s}}^-$,
 and
\[
a_n(p):=(p-1)a_{n-1}(p)+2-q_0,\qquad a_1(p)=1.
\]
%It is clear by induction that $\frac1{q_n}<\frac{q_0-2}{q_0(p-2)}$ which implies that $q_n$ is strictly increasing in $n$. 
Moreover,
\[
a_n(p)=\frac{b_n(p)}{p-2},\quad\text{where }\; b_n(p):=(p-q_0)(p-1)^{n-1}-2+q_0.
\]
Clearly, since $2<p<q_0$, $a_n(p)$ is strictly decreasing in $n$ and $a_n(p)\to -\infty$, as $n \to \infty$.
Therefore, we conclude in a finite number of steps. Indeed, by Young inequality, we deduce that $u\in L^\infty$ whenever 
$$
q_n>\left(\frac{1+s}{1-s}\right)'=\frac{1+s}{2s},\quad\text{ namely }\quad
a_n(p)<\frac{4s}{1-s}.
$$

Next, we notice again from \eqref{integralform} and the fact $\hat{K}_s\in L^\infty$ that
\[
\|u\|_{\dot{H}^{(2,2s)}}
=\|K_s\ast g(u)\|_{\dot{H}^{(2,2s)}}
\lesssim \|g(u)\|_{L^2}\leq\|u\|_{L^\infty}^{p-2}\|u\|_{L^2}<\infty.
\]
Thus, $u\in  H^{(2,2s)}$.

Now, if $2s\geq1$, then $u_y\in L^2$, and we proceed to find  from $|(g(u))_y|\lesssim |u|^{p-2}|u_y|$ that
\begin{equation}\label{formula-6}
\norm{D_y^{2s+1}u}_{L^2}=
\norm{D_y^{2s}K_s\ast D_y(g(u))}_{L^2}
\lesssim\|u_y\|_{L^2}\|u\|_{L^\infty}^{p-2}
	\leq \| u\|_{H^{(2,2s)}}
	\|u\|_{L^\infty}^{p-2}.
\end{equation}
 This shows that $H^{(2,2s+1)}$ is valid when $s\geq1/2$. For the case $s<1/2$, consider the unique integer $N$ such that $1/(N+1)\leq s<1/N$. Then, for any $k=1,\cdots,N$, we have:
\begin{equation}\label{formula-5}
	\|D_y^{(k+1)s}u\|_{L^2}\lesssim \|D_y^{ks} g(u)\|_{L^2}.
\end{equation}
Now, we use the identity (\cite[Lemma 3.1]{fls})
\[
\|D_y^{s}f\|_{L^2(\rr)}^2=C_s\int_{\rr^2}\frac{|f(x)-f(z)|^2}{|x-z|^{1+2s}}\;\dd x\dd z,\qquad s\in(0,1),
\]
and the fact $u\in L^\infty$ to deduce from \eqref{formula-5} that
\[\begin{split}
	\|D_y^{(k+1)s}u\|_{L^2}^{2}&\lesssim
	\int_{\rr^3}\frac{|g(u)(x,y)-g(u)(x,z)|^2}{|y-z|^{1+2ks}}\dd x\dd z\dd y\\
	&\lesssim\int_{\rr^3}\frac{|u(x,y)-u(x,z)|^2}{|y-z|^{1+2ks}}\dd x\dd z\dd y=
	\|D_y^{ks}u\|_{L^2}^{2}.
\end{split}\]
By iteration, we obtain that $u\in H^{(2,(N+1)s)}$. Similar to the case $2s\geq1$, we can now conclude from the fact $(N+1)s\geq1$ that $u\in H^{(2,2s+1)}$. The proof of $u\in H^{(3,0)}$ is analogous to \eqref{formula-6}. The fact that $u\in C_\infty$ is a direct consequence of $u\in H^{(3,2s+1 )}$. Moreover, by the standard argument, we have
\[
\|u_y\|_{L^\infty}\lesssim \|\eta\hat{u}\|_{L^1}
\lesssim
\norm{\frac{|\eta|}{1+|\xi|^3+|\eta|^{2s+1}}}_{L^2}\norm{u}_{H^{(3,2s+1)}}\lesssim
\norm{u}_{H^{(3,2s+1)}},
\]
 provided $s>2/5$. So $u_y\in C_\infty$ if $s>2/5$.
 
Finally, since $u\in C_\infty$ and $K_s, xK_s,yK_s\in\lt$, so $\hat{K}_s\in H^1$, and we   use Lemma 3.1 in \cite{chenbona} and Theorem 5.9 in \cite{epb} together with  Theorem \ref{decay-regul-lemma}  to infer that \eqref{bounds} holds.
	\end{proof}  
	%%%%%%%%%%%%%%%%%%
Thanks to the above regularity results we are able to prove that each solution of \eqref{equ-0} satisfies a Pohozaev identity.
 
	\begin{thm} \label{ph}
		Let $2<p<\frac{2(1+s)}{1-s}$,  $s\in(1/4,1)$ and $u \in H$ be a solution to \blue{\eqref{equ-0}}
  % \red{the equation
		% \begin{align} \label{equ1}
		% 	-\partial_{xx} u + D_y^{2s} u + \alpha u=|u|^{p-2}u \quad \mbox{in} \,\, \R^2,
		% \end{align}}
		then $Q(u)=0$, where $Q$ is defined in \eqref{Q}.
		\end{thm}
	\begin{proof}
Notice first from Theorem \ref{decay-regul-lemma} that $u_x$ and $u_y$ are well-defined. By \eqref{bounds}, if $s>1/4$, we have that $xu, y u \in L^2(\R^2)$. 
 
Repeating the  argument of Theorem \ref{decay-regul-lemma}, we observe that $xu_x,y u_y\in L^2(\R^2)$.

		Multiplying \eqref{equ-0} by $y u_y$ and integrating on $\R^2$, we get that
		\begin{align} \label{ph1}
			-\int_{\R^2} y u_{xx}u_y\dd x\dd y + 
			\int_{\R^2} y u_yD_y^{2s}  u   \dd x\dd y+  \alpha \int_{\R^2}  y uu_y \dd x\dd y= \int_{\R^2} y|u|^{p-2} u u_y   \dd x\dd y.
		\end{align}
	Since $u_x,u_y\in C_\infty$ (see Theorem \ref{decay-regul-lemma}),  we get, by using integration by parts,  that
		$$
		-\int_{\R^2} y u_{xx}    u_y \dd x\dd y=\int_{\R^2} y   u_x u_{xy}   \dd x\dd y=\frac 12 \int_{\R^2} y \partial_{y} |u_x|^2 \dd x\dd y=-\frac 12 \int_{\R^2} |u_x|^2  \dd x\dd y,
		$$
		$$
		\alpha \int_{\R^2}  y u\partial_y u \dd x\dd y=-\frac{\alpha}{2} \int_{\R^2} |u|^2  \dd x\dd y
		$$
		and
		$$
		\int_{\R^2} y|u|^{p-2} u\partial_y u   \dd x\dd y =-\frac 1p \int_{\R^2} |u|^p  \dd x\dd y.
		$$
		Note that
		\begin{align*} %\label{identity1}
			D_y^{2s} (y  u_y)=2s D_y^{2s} u + y   D_y^{2s} u_y.
		\end{align*}
		Moreover, there holds that
		\begin{align*}
			\int_{\R^2} D_y^{2s}  u\left(y \partial_y u\right)  \dd x\dd y&= \int_{\R^2} u D_y^{2s} \left(y \partial_y u\right)  \dd x\dd y \\
			&=2s \int_{\R^2} u D_y^{2s} u  \dd x\dd y +\int_{\R^2} y u \partial_y D_y^{2s} u \dd x\dd y \\
			&=(2s-1) \int_{\R^2} u D_y^{2s} u  \dd x\dd y -\int_{\R^2} u D_y^{2s} \left(y \partial_y u\right)  \dd x\dd y.
		\end{align*}
		This then leads to
		$$
		\int_{\R^2} D_y^{2s}  u\left(y \partial_y u\right)  \dd x\dd y=\frac{2s-1}{2}\int_{\R^2} u D_y^{2s} u  \dd x\dd y =\frac{2s-1}{2}\int_{\R^2} |D_y^{s} u|^2  \dd x\dd y.
		$$
		Coming back to \eqref{ph1}, we then obtain that
		\begin{align} \label{ph2}
			\frac 12 \int_{\R^2} |\partial_ x u|^2  \dd x\dd y +\frac{1-2s}{2}\int_{\R^2} |D_y^{s} u|^2  \dd x\dd y +\frac {\alpha}{2} \int_{\R^2} |u|^2  \dd x\dd y=\frac 1p \int_{\R^2} |u|^p  \dd x\dd y.
		\end{align}
		On the other hand, multiplying \eqref{equ-0} by $x \partial_x u$ and integrating on $\R^2$, we have that
		\begin{align} \label{identity2}
			-\int_{\R^2} x \partial_{xx} u \partial_x u \dd x\dd y + \int_{\R^2} D_y^{2s}  u\left(x \partial_x u\right)  \dd x\dd y+  \alpha \int_{\R^2}  x u\partial_x u \dd x\dd y= \int_{\R^2} x|u|^{p-2} u\partial_x u   \dd x\dd y.
		\end{align}
		Applying the fact that $ D_y^{2s} (x \partial_x u)=x \partial_x D_y^{2s} u$, we are able to similarly derive that
		$$
		-\int_{\R^2} x \partial_{xx} u \partial_x u \dd x\dd y= \frac 12 \int_{\R^2} |\partial_x u|^2  \dd x\dd y, \quad  \int_{\R^2} D_y^{2s}  u\left(x \partial_x u\right)  \dd x\dd y=-\frac 12 \int_{\R^2} |D_y^{s} u|^2  \dd x\dd y
		$$
		and
		$$
		\alpha \int_{\R^2}  x u\partial_x u \dd x\dd y=-\frac{\alpha}{2} \int_{\R^2}|u|^2 \,\dd x\dd y, \quad \int_{\R^2}  x |u|^{p-2}u\partial_x  u \dd x\dd y = -\frac{1}{p} \int_{\R^2}|u|^p  \dd x\dd y.
		$$
		Therefore, from \eqref{identity2}, we get that
		$$
		-\frac 12 \int_{\R^2} |\partial_ x u|^2  \dd x\dd y +\frac{1}{2}\int_{\R^2} |D_y^{s} u|^2  \dd x\dd y +\frac {\alpha}{2} \int_{\R^2} |u|^2  \dd x\dd y=\frac 1p \int_{\R^2} |u|^p  \dd x\dd y.
		$$
		This along with \eqref{ph2} results in
		\begin{align*} %\label{ph3}
			\frac{1-s}{2} \int_{\R^2} |\partial_ x u|^2  \dd x\dd y +\frac{1-s}{2}\int_{\R^2} |D_y^{s} u|^2  \dd x\dd y +\frac {\alpha(s+1)}{2} \int_{\R^2} |u|^2  \dd x\dd y=\frac {s+1}{p} \int_{\R^2} |u|^p  \dd x\dd y.
		\end{align*}
		In addition, by multiplying \eqref{equ-0} by $u$ and integrating on $\R^2$, we conclude that
		$$
		\int_{\R^2} |\partial_ x u|^2  \dd x\dd y +\int_{\R^2} |D_y^{s}u|^2  \dd x\dd y +\alpha \int_{\R^2} |u|^2  \dd x\dd y=\int_{\R^2} |u|^p  \dd x\dd y.
		$$
		As a result, we now obtain that
		$$
		s \int_{\R^2} |\partial_ x u|^2  \dd x\dd y +s\int_{\R^2} |D_y^{s} u|^2  \dd x\dd y =\frac {(s+1)(p-2)}{2p} \int_{\R^2} |u|^p  \dd x\dd y.
		$$
		This means that $Q(u)=0$. Thus the proof is completed.
	\end{proof}

% \torange{The embedding properties and the Pohozaev identity hold for  $p < \frac{2(s+1)}{1-s}$. Is it correct the following part?}
% \tcyan{If I am right, we can rely on the regularity obtained from Theorem \ref{decay-regul-lemma} to get the following identities, otherwise we should pose a regularity assumption on the solutions in the result}
%\torange{I added the assumption in blue. It should be enough...}
\begin{cor} \label{nonexistence}
		If $p \geq \frac{2(s+1)}{1-s}$, then there exists no nontrivial solutions in $H \cap L^p(\R^2) $ to the equation
		\begin{align} \label{equ11}
			-\partial_{xx} u + D_y^{2s} u + \alpha u=|u|^{p-2}u \quad \text{in} \,\, \R^2, \quad \alpha>0.
		\end{align}
	\end{cor}
	\begin{proof}
		If $u \in H$ is a solution to \eqref{equ11}, by Theorem \ref{decay-regul-lemma}, then
		\begin{align} \label{nonexist1}
			s \int_{\R^2} |\partial_ x u|^2  \dd x\dd y +s\int_{\R^2} |D_y^{s} u|^2  \dd x\dd y =\frac {(s+1)(p-2)}{2p} \int_{\R^2} |u|^p  \dd x\dd y.
		\end{align}
		On the other hand, multiplying \eqref{equ11} by $u$ and integrating on $\R^2$, we have that
		\begin{align} \label{nonexist2}
			\int_{\R^2} |\partial_ x u|^2  \dd x\dd y +\int_{\R^2} |D_y^{s} u|^2  \dd x\dd y + \alpha \int_{\R^2}|u|^2 \, dx= \int_{\R^2} |u|^p  \dd x\dd y.
		\end{align}
		Accordingly, from \eqref{nonexist1} and \eqref{nonexist2}, we obtain that
		\begin{align} \label{cor1}
			\alpha \int_{\R^2}|u|^2  \dd x\dd y=\left(1-\frac{(s+1)(p-2)}{2ps}\right) \int_{\R^2}|u|^p  \dd x\dd y.
		\end{align}
		Since $p \geq \frac{2(s+1)}{1-s}$, then
		$$
		1-\frac{(s+1)(p-2)}{2ps} \leq 0.
		$$
		Then we conclude from \eqref{cor1} that $u=0$, because of $\alpha>0$. This completes the proof.
	\end{proof}
	
	\section{Normalized solution in the mass subcritical case}\label{subcritical-section}

	In this section, we consider solutions to \eqref{equ-0}-\eqref{mass} with $s\in (0,1)$ in the mass subcritical case, i.e. $2<p<\frac{2(3s
		+1)}{s+1}$. In this case, from \eqref{gn}, we know that that the energy functional $E$ restricted on $S_c$ is bounded from below. Therefore, we are able to introduce the following minimization problem,
	\begin{align} \label{min}
		m(c):=\inf_{u \in S_c} E(u).
	\end{align}
	Clearly, minimizers to \eqref{min} are solutions to \eqref{equ-0}-\eqref{mass}.
	
	\begin{thm} \label{thm1}
		Let $2<p<\frac{2(3s+1)}{s+1}$, then any minimizing sequence to \eqref{min} is compact in $H$, up to translations, for any $c>0$. In particular, there exist minimizers to \eqref{min}, for any $c>0$.
	\end{thm}
	\begin{proof}
		Suppose that $\{u_n\} \subset S_c$ is a minimizing sequence to \eqref{min}. In view of \eqref{gn}, we see that
		\begin{align} \label{below1}
			\begin{split}
				E(u_n)& \geq \frac 12 \int_{\R^2} |\partial_x u_n|^2 \,\dd x\dd y + \frac 12 \int_{\R^2} |D_y^{s}u_n|^2  \dd x\dd y \\
				& \quad -\frac{C_{p,s} c^{\frac p 2-\frac{(p-2)(1+s)}{4s}}}{p}\left(\int_{\R^2} |\partial_x u_n|^2  \dd x\dd y \right)^{\frac{p-2}{4}}\left(\int_{\R^2} |D_y^{s} u_n|^2  \dd x \dd y \right)^{\frac{p-2}{4s}}.
			\end{split}
		\end{align}
		Since $2 <p<\frac{2(3s+1)}{s+1}$, then
		$$
 0<\frac{p-2}{4} +\frac{p-2}{4s}<1
		$$
		It then follows, from \eqref{below1}, that $\{u_n\} \subset S_c$ is bounded in $H$. In addition, by using \eqref{scaling-0} and \eqref{scaling}, we are able to get that $E(u_t)<0$, for any $t>0$ small enough. This suggests that $m(c)<0$ for any $c>0$. Thanks to Lemma \ref{concentration}, then there exist a sequence $\{z_n\} \subset \R^2$ and a nontrivial $u \in H$ such that $u_n(\cdot-z_n) \wto u$ in $H$ as $n \to \infty$. Furthermore, there holds that $m(\theta c) \leq \theta m(c)$, for any $c>0$ and $\theta>1$. Indeed, by the definition of $m(c)$, one has that, for any $\eps>0$, there exists $u \in S_c$ such that $E(u) \leq m(c) + \eps$. Note that $\sqrt{\theta} u \in S(\theta c)$, $\theta>1$ and $p>2$, then
		$$
		m(\theta c) \leq E(\sqrt{\theta} u)=\frac {\theta}{2} \|\partial_x u\|_2^2  + \frac {\theta}{2} \|D_y^{s} u\|_2^2   
-\frac {\theta^{\frac p 2}}{p} \|u\|_p^{p} \leq \theta E(u)\leq \theta(m(c) + \eps).
		$$
		Thus we have the desired conclusion. 

We claim that $m(\theta c) <\theta m(c)$, for any $c>0$ and $\theta>1$. Suppose, by contradiction, that $m(\theta c) =\theta m(c)$ for some $c>0$ and $\theta>1$. Let $\{u_n\}\subset S_c$ be a minimizing sequence to \eqref{min}, then 
\[
o_n(1)=m(\theta c)-\theta E(u_n)
\le \frac {\theta-\theta^{\frac p 2}}{p} \|u_n\|_p^{p} \le 0.
\]
This implies that $u_n \to 0$ in $L^p(\R^2)$ and so $m(c)\ge 0$, reaching a contradiction. At this point, we can utilize the Lions concentration-compactness principle in \cite{Li1, Li2} to complete the proof.
	\end{proof}
	
	\begin{thm} \label{thm2}
		Let $2<p<\frac{2(3s+1)}{s+1}$ and define
		$$
		G(c):=\{ u \in S_c : E(u)=m(c)\},
		$$
		then $G(c) \neq \emptyset$ and $G(c)$ is orbitally stability, for any $c>0$.
	\end{thm}
	\begin{proof}
		In virtue of Theorem \ref{thm1}, we first have that $G(c) \neq \emptyset$ for any $c>0$. Let us now show that $G(c)$ is orbitally stability for any $c>0$. Suppose by contradiction that $G(c)$ is unstable for some $c>0$. Then, with the help of the conservation laws, there exists a sequence $\{u_n\} \subset S_c$ such that $m(c)=E(u_n) +o_n(1)$. According to Theorem \ref{thm1}, we know that $\{u_n\} \subset S_c$ is compact in $H$ up to translations. We then reach a contradiction. Thus the proof is completed.
	\end{proof}

% \torange{I would delete the red part. If you agree, please do it}
	\begin{prop}
		Let $u \in S_c$ be a minimizer to \eqref{min}, then there exists a constant $\theta \in \mathbb{S}^1$ such that $u(x, y)=e^{\textnormal{i} \theta} |u(x, y)|$ for $(x, y) \in \R^2$. Moreover, $|u|$ is  non-negative,  partially (axially) symmetric and strictly decreasing with respect to $x ,y$.
	\end{prop}
	\begin{proof}
		In view of \cite[Theorem 7.8]{LL} and \cite[Theorem 7.12]{LL}, we know that
		$$
		\int_{\R^2} |\partial_x |u||^2 \,\dd x\dd y  \leq  \int_{\R^2} |\partial_x u|^2 \,\dd x\dd y, \quad \int_{\R^2} |D_y^{s}|u||^2 \dd x\dd y \leq \int_{\R^2} |D_y^{s} u|^2  \dd x\dd y.
		$$
		As a consequence, we have that $E(|u|)=m(c)$ and $E(|u|)=E(u)$. This, in turn, implies that
		$$
		\int_{\R^2} |\partial_x |u||^2 \,\dd x\dd y  =  \int_{\R^2} |\partial_x u|^2 \,\dd x\dd y, \quad \int_{\R^2} |D_y^{s}|u||^2  \dd x\dd y = \int_{\R^2} |D_y^{s} u|^2  \dd x\dd y.
		$$
		Thus we derive that there exists a constant $\theta \in \mathbb{S}^1$ such that $u=e^{\textnormal{i} \theta}|u|$.
 The remaining proof comes directly from \cite[Theorem 1.3]{FW}. This completes the proof.
	\end{proof}

	\section{Normalized solution in the mass critical case}\label{critical-section}

	In this section, we investigate solutions to \eqref{equ-0}-\eqref{mass} for $s\in (1/4,1)$ in the mass critical case, i.e. $p=\frac{2(3s+1)}{s+1}$. The  result reads as follows.

	\begin{thm}  \label{thm1-mass-cr}
		Let $p=\frac{2(3s+1)}{s+1}$, then there exists a constant $c_*>0$ such that 
		\begin{align*}
			m(c)= \left\{
			\begin{aligned}
				&0, &\quad 0 <c \leq c_*, \\
				&-\infty, &\quad c > c_*,
			\end{aligned}
			\right.
		\end{align*}
		where 
		$$
		c_*:=\left(\frac{3s+1}{C_H(s+1)}\right)^{\frac{s+1}{2s}}
		$$
		and $C_H>0$ is the optimal constant in \eqref{gn-hH} for $p=\frac{2(3s+1)}{s+1}$. In addition, there exists no solution in $H$ to \eqref{equ-0}-\eqref{mass} for any $0<c<c_*$.
	\end{thm}
	\begin{proof}
		Due to \eqref{scaling}, for any $u \in S_c$, we first see that $E(u_t) \to 0$, as $t \to 0$. This shows that $m(c) \leq 0$, for any $c>0$. On the other hand, by invoking   \eqref{gn-hH}, we have that
 		\begin{align*}
			E(u) &\ge \frac 12\left(\int_{\R^2} |\partial_x u|^2 \,\dd x\dd y +\int_{\R^2} |D_y^{s}u|^2  \dd x\dd y\right)
\\
&\qquad-\frac{C_H(s+1) c^{\frac{2s}{s+1}}}{2(3s+1)}\left(\int_{\R^2} |\partial_x u|^2  \dd x\dd y +\int_{\R^2} |D_y^{s} u|^2  \dd x\dd y \right)
\\&=\frac 12 \left(1-\frac{C_H(s+1)c^{\frac{2s}{s+1}}}{3s+1}\right)\left(\int_{\R^2} |\partial_x u|^2 \,\dd x\dd y +\int_{\R^2} |D_y^{s}u|^2  \dd x\dd y\right).
		\end{align*}
 
		Hence there holds that $E(u) \geq 0$, for any $0<c \leq c_*$. Consequently, we derive that $m(c)=0$, for any $0<c \leq c_*$. 

We now prove that $m(c)=-\infty$, for any $c>c_*$. Let $u \in H$ be an optimal function such that   \eqref{gn-hH}  with $p=\frac{2(3s+1)}{s+1}$ holds (see Corollary \ref{cor-gn}). Define
		$$
		w:= \frac{u}{\|u\|_{L^2}} c^{\frac 12} \in S_c.
		$$
		Observe from \eqref{scaling-0} and \eqref{scaling}  that
 
\begin{align*}
			E(w_t)&=\frac{t^{2s}}{2} \int_{\R^2} |\partial_x w|^2  \dd x\dd y + \frac{t^{2s}}{2} \int_{\R^2} |D_y^{2s} w|^2  \dd x\dd y-\frac{(s+1)t^{2s}}{2(3s+1)} \int_{\R^2}|w|^{\frac{2(3s+1)}{s+1}}  \dd x\dd y \\
			&=\frac{ct^{2s}}{2\|u\|_{L^2}^2} \left(\int_{\R^2} |\partial_x u|^2  \dd x\dd y + \int_{\R^2} |D_y^{s} u|^2  \dd x\dd y\right)-\frac{(s+1)c^{\frac{3s+1}{s+1}}t^{2s}}{2(3s+1)\|u\|_{L^2}^{\frac{2(3s+1)}{s+1}}} \int_{\R^2}|u|^\frac{2(3s+1)}{s+1}  \dd x\dd y \\
			&=\frac{ct^{2s}}{2\|u\|_{L^2}^2} \left(1-\frac{C_H(s+1)c^{\frac{2s}{s+1}}}{3s+1}\right)\left(\int_{\R^2} |\partial_x u|^2  \dd x\dd y + \int_{\R^2} |D_y^{s} u|^2  \dd x\dd y\right),
		\end{align*}
		from which we obtain that $E(w_t) \to -\infty$, as $t \to \infty$, for any $c>c_*$. This indicates that $m(c)=-\infty$ for any $c>c_*$.

  Finally, we show that there is no solution to \eqref{equ-0}-\eqref{mass} for any $0<c < c_*$. Suppose that $u \in S_c$ were a solution to \eqref{equ-0}-\eqref{mass} for some $0<c<c_*$. By Theorem \ref{ph}, then $Q(u)=0$, i.e.
		\begin{align*}
			s \int_{\R^2} |\partial_ x u|^2  \dd x\dd y +s\int_{\R^2} |D_y^{s} u|^2  \dd x\dd y &=\frac {s(s+1)}{3s+1} \int_{\R^2} |u|^{\frac{2(3s+1)}{s+1}}  \dd x\dd y \\
			& \leq \frac{C_H s  c^{\frac{2s}{s+1}}}{3s+1} \left( \int_{\R^2} |\partial_ x u|^2  \dd x\dd y +\int_{\R^2} |D_y^{s} u|^2  \dd x\dd y\right),
		\end{align*}
		where we also used \eqref{gn-hH} for the inequality. Hence we get that $u=0$ for $0<c<c_*$. This is impossible. The desired conclusion follows. Thus the proof is completed.
	\end{proof}

	\section{Normalized solution in the mass supercritical case}\label{supercritical-section}

	In this section, we shall investigate the mass supercritical case with a general nonlinearity  and assuming that $s\in(1/4,1)$. Here we follow some ideas of \cite{jeanjeanlu2} and we will skip some details. Anyway, the presence of the anisotropic operator creates several technical difficulties. More precisely we consider 	
\begin{enumerate}
		\item[(H0)] $f:\mathbb{R}\rightarrow \mathbb{R}$ is continuous;
		\item[(H1)] $\lim\limits_{t\rightarrow 0}\frac{f(t)}{\vert t \vert^{2^{\sharp}-1}}=0$;
		\item[(H2)] $\lim\limits_{t\rightarrow \infty}\frac{f(t)}{\vert t \vert^{2^{*}-1}}=0$;
		\item[(H3)] $\lim\limits_{t\rightarrow \infty}\frac{F(t)}{\vert t \vert^{2^{\sharp}}}=+\infty$;
		\item[(H4)] $t\mapsto \frac{\tilde{F}(t)}{\vert t \vert^{2^{\sharp}}}$ is strictly decreasing on $(-\infty,0)$ and strictly increasing on $(0,\infty)$;
		\item[(H5)] $f(t)t<2^{*}F(t)$ for all $t\in \mathbb{R}\setminus \{0\}$;
	\end{enumerate}
	where $F(t):=\int_{0}^{t} f(\tau) \,d\tau$, $\tilde{F}(t):=f(t)t-2F(t)$, $2^{\sharp}=\frac{2(3s+1)}{s+1}$ and $2^{*}=\frac{2(s+1)}{1-s}$.
We deal with the following equation
	\begin{align*}
		-\partial_{xx} u + D_y^{2s} u + \alpha u=f(u) \qquad \text{in }\R^2.
	\end{align*}
	The associated energy functional is given by
	\begin{align*}
		E(u):=\frac 12 \|\partial_x u\|_2^2  + \frac 12 \|D_y^{s} u\|_2^2  - \int_{\R^2} F(u) \dd x\dd y,
	\end{align*}
	while the Nehari-Pohozaev identity is 
	\begin{align*}
		Q(u):= s\|\partial_x u\|_2^2 + s\|D_y^{s} u\|_2^2  -\frac{s+1}{2} \int_{\R^2} \tilde{F}(u) \dd x\dd y.
	\end{align*}
	Moreover, we need the following notations:
	\begin{enumerate}
		\item[(i)] $S_c:=\{u\in H: \|u\|_{2}^2=c\}$ and $B_c:=\{u\in H: \|u\|_{2}^2\leq c\}$.
	\end{enumerate}
	\subsection{Preliminary results}
	\begin{lem}\label{lm:estimate_energy}
		Assume that $(H0)-(H2)$ are true. Then the following statements hold.
		\begin{enumerate}
			\item[(i)] For any $c>0$, there exists $\delta:=\delta(c)>0$ small enough such that
			$$
			\frac{1}{4}\|\partial_x u\|_2^2 + \frac 14 \|D_y^{s} u\|_2^2  \leq E(u) \leq \|\partial_x u\|_2^2 +  \|D_y^{s} u\|_2^2, 
			$$
			for all $u \in B_c$ with $\|\partial_x u\|_2^2 +  \|D_y^{s} u\|_2^2 \leq \delta^2$.
			\item[(ii)]
			For any $c>0$, there exists $\delta:=\delta(c)>0$ small enough such that
			$$
			Q(u)\geq \frac{s}{2}\|\partial_x u\|_2^2 + \frac s2 \|D_y^{s} u\|_2^2,
			$$
			for all $u \in B_c$ with $\|\partial_x u\|_2^2 +  \|D_y^{s} u\|_2^2 \leq \delta^2$.
			\item[(iii)] If $\left\{u_n\right\}\subset H$ is a bounded sequence  satisfying $\lim\limits_{n \rightarrow \infty}\left\|u_n\right\|_{L^{2^{\sharp}}}=0$, then
			$$
			\lim _{n \rightarrow \infty} \int_{\mathbb{R}^2} F\left(u_n\right) \dd x\dd y=\lim _{n \rightarrow \infty} \int_{\mathbb{R}^2} \widetilde{F}\left(u_n\right) \dd x\dd y=0.
			$$
			\item[(iv)] If $\left\{u_n\right\},\left\{v_n\right\} \subset H$ are bounded sequences  with $\lim \limits_{n \rightarrow \infty}\left\|v_n\right\|_{L^{2^{\sharp}}}=0$, then
			$$
			\lim _{n \rightarrow \infty} \int_{\mathbb{R}^2} f\left(u_n\right) v_n\dd x\dd y=0.
			$$
		\end{enumerate}
	\end{lem}
	\begin{proof}
		\begin{enumerate}
			\item[(i)]It follows from $(H0)-(H2)$ and the definition of $F(t)$ that, for any $\epsilon>0$, there exists a constant $C_{\epsilon}>0$ such that $\vert F(t) \vert \leq \epsilon \vert t \vert^{2^{\sharp}} + C_{\epsilon}\vert t \vert^{2^*}$ holds, for all $t\in \R$.
			For any $u\in B_c$, according to the Gagliardo-Nirenberg  inequality, we have
			\begin{align*}
				\int_{\R^2}|F(u)|\,\dd x\dd y\leq &\epsilon\int_{\R^2}|u|^{2^{\sharp}}\,\dd x\dd y + C_{\epsilon} \int_{\R^2}|u|^{2^{*}}\,\dd x\dd y\\
				\leq &\epsilon C_{2^{\sharp},s}\left(\int_{\R^2}|u|^2  \dd x\dd y \right)^{\frac{2s}{s+1}} \left(\int_{\R^2} |\partial_x u|^2  \dd x\dd y \right)^{\frac{s}{s+1}}\left(\int_{\R^2} |D_y^{s} u|^2  \dd x\dd y \right)^{\frac{1}{s+1}}\\
				&+C_{\epsilon}C_{2^*,s}\left(\int_{\R^2} |\partial_x u|^2  \dd x\dd y \right)^{\frac{s}{-s+1}}\left(\int_{\R^2} |D_y^{s} u|^2  \dd x\dd y \right)^{\frac{1}{-s+1}}\\
				\leq & \left[\epsilon C_{2^{\sharp},s} c^{\frac{2s}{s+1}}+C_{\epsilon}C_{2^{*},s}\left(\int_{\R^2} |\partial_x u|^2  \dd x\dd y \right)^{\frac{2s^2}{1-s^2}}\left(\int_{\R^2} |D_y^{s} u|^2  \dd x\dd y \right)^{\frac{2s}{1-s^2}}\right]\\
				&\times \left(\int_{\R^2} |\partial_x u|^2  \dd x\dd y \right)^{\frac{s}{s+1}}\left(\int_{\R^2} |D_y^{s} u|^2  \dd x\dd y \right)^{\frac{1}{s+1}}.
			\end{align*}
			By   Young's inequality, it holds that 
			\begin{align*}
				\int_{\R^2}|F(u)|\,\dd x\dd y \leq & \left[\epsilon C_{2^{\sharp},s} c^{\frac{2s}{s+1}}+C_{\epsilon}C_{2^{*},s}\left(\int_{\R^2} |\partial_x u|^2  \dd x\dd y \right)^{\frac{2s^2}{1-s^2}}\left(\int_{\R^2} |D_y^{s} u|^2  \dd x\dd y \right)^{\frac{2s}{1-s^2}}\right]\\
				&\times  \frac 1{s+1}\left(\int_{\R^2} |\partial_x u|^2  \dd x\dd y +\int_{\R^2} |D_y^{s} u|^2  \dd x\dd y \right).
			\end{align*}
			As a result, we can choose $\epsilon$ and $\delta$ sufficiently small to prove $(i)$.
			\item[(ii)] The proof is similar to (i) and we omit the details.
			\item[(iii)]
			It follows from $(H0)-(H2)$ and the definition of $F(t)$ that for any $\epsilon>0$, there exists a constant $C_{\epsilon}>0$ such that $\vert f(t) \vert \leq C_{\epsilon} \vert t \vert^{2^{\sharp}-1} + \epsilon\vert t \vert^{2^*-1}$ and $\vert F(t) \vert \leq C_{\epsilon} \vert t \vert^{2^{\sharp}} + \epsilon\vert t \vert^{2^*}$ hold for all $t\in \R$.
			Consequently, we have $\vert \tilde{F}(t) \vert \leq C_{\epsilon} \vert t \vert^{2^{\sharp}} + \epsilon\vert t \vert^{2^*}$ hold for all $t\in \R$, and so 
			\begin{align*}
				\lim_{n \rightarrow \infty}\int_{\mathbb{R}^2} \vert F\left(u_n\right)\vert \,\dd x\dd y \leq & C_{\epsilon}\lim_{n \rightarrow \infty}\int_{\R^2}|u_n|^{2^{\sharp}}\,\dd x\dd y +\epsilon \lim_{n \rightarrow \infty}\int_{\R^2}|u_n|^{2^{*}}\,\dd x\dd y\\
				\leq & \epsilon \lim_{n \rightarrow \infty}\|u_n\|^{2^{*}}\leq \epsilon C,
			\end{align*}
			where we have used the fact that $H$ embeds into $L^p(\mathbb{R}^2)$ continuously for $p\in [2,2^*]$.
			\item[(iv)] The proof is similar to $(ii)$ and we omit the details.
		\end{enumerate}
	\end{proof}
	\begin{lem}\label{lm:limit_scaling}
		Assume that $(H 0)-(H 3)$ hold. For any $u \in H \backslash\{0\}$, it holds that
		\begin{enumerate}
			\item[(i)] $E(u_t) \rightarrow 0^{+}$as $t \rightarrow 0^{+}$;
			\item[(ii)] $E(u_t) \rightarrow-\infty$ as $t \rightarrow+\infty$.
		\end{enumerate}
	\end{lem}
	\begin{proof}
		\begin{enumerate}
			\item[(i)] Let $u\in H\backslash\{0\}$, then $u\in B_c$ with $c=\|u\|_{L^2}^2$. So applying  Lemma \ref{lm:estimate_energy}(i) to $u_t$, with $t>0$ sufficiently small, it follows that
			$$
			\frac{t^{2s}}{4}\|\partial_x u\|_2^2 + \frac{t^{2s}}{4}\|D_y^{s} u\|_2^2   \leq E(u_t) \leq t^{2s}\|\partial_x u\|_2^2 +  t^{2s}\|D_y^{s} u\|_2^2 
			$$
			and (i) is clear.
			\item[(ii)]For any $\lambda \geq 0$, we define a function $h_\lambda: \mathbb{R} \rightarrow \mathbb{R}$ as follows:
			$$
			h_\lambda(t):=\left\{\begin{aligned}
				\frac{F(t)}{|t|^{2^{\sharp}}}+\lambda, &\qquad \text { for } t \neq 0, \\
				\lambda, & \qquad\text { for } t=0 .
			\end{aligned}\right.
			$$
			It is clear that $h_\lambda$ is continuous and
			$$
			h_\lambda(t) \rightarrow+\infty, \quad \text { as } t \rightarrow \infty.
			$$
			Choose $\lambda>0$ large enough such that $h_\lambda(t) \geq 0$ for any $t \in \mathbb{R}$.  According to the Fatou's lemma, it follows that
			$$
			\lim _{t \rightarrow+\infty} \int_{\mathbb{R}^2} h_\lambda\left(t^{\frac{s+1}{2}} u\right)|u|^{2^{\sharp}} \dd x\dd y=+\infty.
			$$
			Consequently, we arrive at
			$$
			\begin{aligned}
				E(u_t)  &=\frac{1}{2}\|\partial_x u_t\|_2^2 + \frac 12 \|D_y^{s} u_t\|_2^2  +\lambda \|u_t\|^{2^\sharp}_{2^\sharp} -\int_{\mathbb{R}^2} h_\lambda(u_t)|u_t|^{2^\sharp} \dd x\dd y \\
				 &=t^{2 s}\left[\frac{1}{2}\|\partial_x u\|_2^2 + \frac 12 \|D_y^{s} u\|^2_2 +\lambda \|u\|_{2^\sharp}^{2^\sharp} -\int_{\mathbb{R}^N} h_\lambda\left(t^{\frac{s+1}{2}} u\right)|u|^{2^\sharp} \dd x\dd y\right]
				\rightarrow& -\infty, \quad \text{as } t \rightarrow \infty.
			\end{aligned}
			$$
		\end{enumerate}
	\end{proof}
	The following lemma is essentially contained in \cite{jeanjeanlu2}  and we omit the proof.
	\begin{lem}\label{lm:nonlinearity}
		If  $(H 0),(H 1),(H 3)$ and $(H 4)$ are true, then we have
		$$
		f(t) t>2^{\sharp} F(t)>0, \quad \text { for all } t \neq 0.
		$$
	\end{lem}
	\begin{lem}\label{lm:maximum_point}
		Assume that  $(H 0)-(H 4)$ hold. For any $u \in H  \backslash\{0\}$, we have
		\begin{enumerate}
			\item[(i)] there exists a unique number $t_u >0$ such that $Q(u_{t_u})=0$;
			\item[(ii)]  $E(u_{t_u})>E(u_t)$, for any $t \neq u_{t_u}$; in particular, $E(u_{t_u})>0$;
			\item[(iii)]  the mapping $u \mapsto t_u$ is continuous in $u \in H \backslash\{0\}$;
			\item[(iv)]  it holds that $t_{u(\cdot+z)}=t_u$, for any $z \in \mathbb{R}^2$. 
		\end{enumerate}
	\end{lem}
	\begin{proof}
		\begin{enumerate}
			\item[(i)] A direct calculation shows that 
			\begin{align*}% \label{eq:scaling}
				E(u_t)=\frac{t^{2s}}{2} \|\partial_x u\|_2^2   + \frac{t^{2s}}{2} \|D_y^{s} u\|_2^2 -t^{-s-1} \int_{\R^2}F(t^{\frac{s+1}{2}}u)  \dd x\dd y
			\end{align*}
			and $t\mapsto E(u_t)$ belongs to $C^1(\R,\R)$ with
			\begin{align*}
				\frac{d}{dt}E(u_t)=\frac 1tQ(u_t).
			\end{align*}
			By Lemma \ref{lm:limit_scaling}, there exists a local maximum point $t_u>0$ such that
			\begin{align*}
				\frac{d}{dt}E(u_{t})_{\vert_{t=t_u}}=\frac{1}{t_u}Q(u_{t_u})=0.
			\end{align*}
			To prove the uniqueness, suppose by contradiction that there exist $0<\tilde{t}_u<t_u$ such that $Q(u_{t_u})=Q(u_{\tilde{t}_u})=0$, that is,
			\begin{align*}
				st_u^{2s}\|\partial_x u\|_2^2  + st_u^{2s}\|D_y^{s} u\|_2^2 -\frac{s+1}{2} t_u^{-s-1}\int_{\R^2} \tilde{F}(t_u^{\frac{s+1}{2}}u) \dd x\dd y=0,
			\end{align*}
			and 
			\begin{align*}
				s\tilde{t}_u^{2s}\|\partial_x u\|^2_2 + s\tilde{t}_u^{2s}\|D_y^{s} u\|_2^2 -\frac{s+1}{2} \tilde{t}_u^{-s-1}\int_{\R^2} \tilde{F}(\tilde{t}_u^{\frac{s+1}{2}}u) \dd x\dd y=0.
			\end{align*}
			As a result, it holds that 
			$$
			\int_{\mathbb{R}^2}\left(\frac{\tilde{F}\left(t_u^{\frac{s+1}{2}} u\right)}{\left|t_u^{\frac{s+1}{2}} u\right|^{2^{\sharp}}}-\frac{\tilde{F}\left(\tilde{t}_u^{\frac{s+1}{2}} u\right)}{\left|\tilde{t}_u^{\frac{s+1}{2}} u\right|^{2^{\sharp}}}\right)|u|^{2^{\sharp}} \dd x\dd y=0.
			$$
			Observe that $\frac{\tilde{F}(s)}{|s|^{2^{\sharp}}}$ is continuous on $\mathbb{R}$ and it is strictly decreasing on $(-\infty, 0)$ and strictly increasing on $(0,+\infty)$, which implies that
			$$
				\int_{\mathbb{R}^2}\left(\frac{\tilde{F}\left(t_u^{\frac{s+1}{2}} u\right)}{\left|t_u^{\frac{s+1}{2}} u\right|^{2^{\sharp}}}-\frac{\tilde{F}\left(\tilde{t}_u^{\frac{s+1}{2}} u\right)}{\left|\tilde{t}_u^{\frac{s+1}{2}} u\right|^{2^{\sharp}}}\right)|u|^{2^{\sharp}} \dd x\dd y>0,
			$$
			a contradiction, and so $\tilde{t}_u=t_u$. The case $\tilde{t}_u>t_u$ is similar.
			\item[(ii)] The proof in (i) implies that $t_u$ is a global maximum point and $E(u_{t_u})>\lim\limits_{t \rightarrow 0^+} E(u_t)=0$.
			\item[(iii)] The mapping $u \mapsto t_u$ is well-defined by (i). 
			Let $\left\{u_n\right\} \subset H$ be a sequence such that $u_n \rightarrow u$ in $H$.
			We first prove the boundedness of $\left\{t_{u_n}\right\}$ via contradiction. 
			Suppose that $t_{u_n} \rightarrow+\infty$, as $n\rightarrow +\infty$. 
			Set $t_n:=t_{u_n}$ and $v_n(x,y):=t_n^{\frac{s+1}{2}}u_n(t_n^s x, t_n y)$. Then we have
% \red{			$$
% 			\begin{aligned}
% 				0 \leq t_{n}^{-2s} E\left(v_n\right) & =\frac 12 \int_{\R^2} |\partial_x v_n|^2 \,\dd x\dd y + \frac 12 \int_{\R^2} |D_y^{s} v_n|^2  \dd x\dd y - t_n^{-3s-1}\int_{\R^2} F(t_n^{\frac{s+1}{2}}v_n) \,\dd x\dd y, \\
% 				& = \frac 12 \int_{\R^2} |\partial_x v_n|^2 \,\dd x\dd y + \frac 12 \int_{\R^2} |D_y^{s} v_n|^2  \dd x\dd y - \int_{\R^2} h_{0}(t_n^{\frac{s+1}{2}}v_n)\vert v_n \vert^{2^{\sharp}}\,\dd x\dd y,\\
% 				&\rightarrow -\infty,\quad \text{as} \quad n \rightarrow \infty,
% 			\end{aligned}
% 			$$}
			$$
			\begin{aligned}
				0 \leq t_{n}^{-2s} E\left(v_n\right) & =\frac 12 \|\partial_x u_n\|_2^2  + \frac 12 \|D_y^{s} u_n\|^2_2  - t_n^{-3s-1}\int_{\R^2} F(t_n^{\frac{s+1}{2}}u_n) \dd x\dd y, \\
				& = \frac 12 \|\partial_x u_n\|_2^2  + \frac 12 \|D_y^{s} u_n\|_2^2  - \int_{\R^2} h_{0}(t_n^{\frac{s+1}{2}}u_n)\vert u_n \vert^{2^{\sharp}}\dd x\dd y,\\
				&\rightarrow -\infty,\quad \text{as} \ n \rightarrow \infty,
			\end{aligned}
			$$
			which is absurd, where we have used Lemma \ref{lm:nonlinearity} and the proof in Lemma \ref{lm:limit_scaling}-(ii).
			Therefore, $\left\{t_{n}\right\}$ is bounded, which implies that there exists $t^* \geq 0$ such that, up to a subsequence, $t_{n} \rightarrow t^*$. 
			Since $u_n \rightarrow u$ in $H$ and $Q(v_n)=0$, we get $v_n\rightarrow u_{t^{*}}$ in $H$ and $Q(u_{t^{*}})=0$.
			The uniqueness implies that $t^*=t_u$.
			\item[(iv)] It follows from the definition and a direct calculation and we omit the proof.
		\end{enumerate}
	\end{proof}
Recalling the definition of $P_c$ (see \eqref{Pc}), the following useful properties hold. 
 \begin{lem}\label{lm:coercive}
		Assume that  $(H 0)-(H 4)$ hold. Then
		\begin{enumerate}
			\item[(i)] $P_c \neq \emptyset$;
			\item[(ii)] $\inf\limits_{u \in P_c}\left(\|\partial_x u\|_2^2  + \|D_y^{s}u\|_2^2\right)>0$;
			\item[(iii)] $\gamma(c):=\inf\limits_{u \in P_c} E(u)>0$;
			\item[(iv)] $E$ is coercive on $P_c$, that is, $E\left(u_n\right) \rightarrow+\infty$ for any $\left\{u_n\right\} \subset P_c$ with $\left\|u_n\right\| \rightarrow \infty$.
		\end{enumerate}
	\end{lem}
	\begin{proof}
		\begin{enumerate}
			\item[(i)] Clear.
			\item[(ii)] Suppose by contradiction that there exists a sequence $\{u_n\}\subset P_c$ such that 
			\begin{align*}
				\|\partial_x u_n\|_2^2 + \|D_y^{s}u_n\|_2^2\rightarrow 0\quad \text{as}\quad n \rightarrow \infty.
			\end{align*}
			According to $Q(u_n)=0$ and Lemma \ref{lm:estimate_energy}(ii), it follows that 
			\begin{align*}
				0=Q(u_n)\geq \frac{s}{2}\|\partial_x u_n\|_2^2  + \frac s2 \|D_y^{s} u_n\|_2^2 >0,
			\end{align*}
			which is absurd, and so (ii) holds.
			\item[(iii)] Suppose by contradiction that there exists a sequence $\{u_n\}\subset P_c$ such that $\lim\limits_{n\rightarrow\infty}E(u_n)=0$.
			According to Lemma \ref{lm:nonlinearity} and $Q(u_n)=0$, it follows that 
			\begin{align*}
				E(u_n)=E(u_n)-\frac{1}{2s}Q(u_n)=\frac{s+1}{4s}\int_{\mathbb{R}^2}\left(f(u_n)u_n-2^{\sharp}F(u_n)\right)\dd x\dd y>0,
			\end{align*}
			a contradiction, and so (ii) holds.
			\item[(iv)] Suppose by contradiction that there exists a sequence $\{u_n\}\subset P_c$ with $\left\|u_n\right\| \rightarrow \infty$ such that $E(u_n)\leq C$ uniformly, for some constant $C>0$.
			Since $\|u_n\|_{2}^2=c$, we have 
			\begin{align*}
				a_n:=\|\partial_x u_n\|_2^2  + \|D_y^{s}u_n\|_2^2\rightarrow \infty,\quad \text{as}\quad n \rightarrow \infty.
			\end{align*}
			Set $t_n:=a_n^{-\frac{1}{2s}}$ and $v_n(x,y):=t_n^{\frac{s+1}{2}}u_n(t_n^s x, t_n y)$.
			It is clear that $\{v_n\}\subset S_c$ and 
			\begin{align*}
				\|\partial_x v_n\|_2^2  + \|D_y^{s}v_n\|^2_2=1,
			\end{align*}
			which implies that $\{v_n\}$ is bounded in $H$.
			Let 
			$$
			\rho:=\limsup _{n \rightarrow \infty}\left(\sup _{z \in \mathbb{R}^2} \int_{B(z, 1)}\left|v_n\right|^2 \dd x\dd y\right).
			$$
			We first consider the nonvanishing case, that is, $\rho>0$.
			Up to a subsequence, there exists a sequence $\left\{z_n\right\} \subset \mathbb{R}^2$ and $w \in H \backslash\{0\}$ such that
			$$
			w_n:=v_n\left(\cdot+z_n\right) \rightharpoonup w \quad \text { in } H \quad \text { and } \quad w_n \rightarrow w \quad \text { a.e. in } \mathbb{R}^2 .
			$$
			Similar to proof in Lemma \ref{lm:maximum_point}-(iii), we can prove that 
\begin{align*}
	0 \leq t_{n}^{2s} E\left(u_n\right) 
&=\frac 12 \|\partial_x v_n\|_2^2 + \frac 12 \|D_y^{s} v_n\|_2^2 - \int_{\R^2} h_{0}(t_n^{-\frac{s+1}{2}}v_n)\vert v_n \vert^{2^{\sharp}}\dd x\dd y
\\
&=\frac 12 - \int_{\R^2} h_{0}(t_n^{-\frac{s+1}{2}}w_n)\vert w_n \vert^{2^{\sharp}}\dd x\dd y\rightarrow -\infty,\quad \text{as }  n \rightarrow \infty,
\end{align*}
% \red{			\begin{align*}
% 				0 \leq t_{n}^{-2s} E\left(w_n\right) \rightarrow -\infty,\quad \text{as} \quad n \rightarrow \infty,
% 			\end{align*}}
			which is absurd. 
			Then we shall consider the vanishing case, that is, $\rho=0$. 
			We first prove a similar result about Lion's vanishing Lemma.
			For $p\in (2,2^{*})$, by H\"{o}lder's inequality and anisotropic Gagliardo-Nirenberg's inequality, we have
			\begin{align*}
				\|v_n\|_{L^p(B(y,r))}\leq \|v_n\|_{L^2(B(y,r))}^{\lambda}\|v_n\|_{L^{2^*}}^{1-\lambda}\leq C\|v_n\|_{L^2(B(y,r))}^{\lambda}\|v_n\|^{1-\lambda},
			\end{align*}
			where $\lambda\in (0,1)$ satisfies $\frac{1}{p}=\frac{\lambda}{2}+\frac{1-\lambda}{2^*}$.
			As a result, there exists a constant $C_1>0$ such that 
			\begin{align*}
				\|v_n\|_{L^p(\mathbb{R}^N)}^p\leq C_1C^p (\sup_{y\in \mathbb{R}^2}\int_{B(y,r)}\vert v_n\vert^2 \dd x\dd y)^{\frac{p\lambda}{2}}\|v_n\|^{p(1-\lambda)},
			\end{align*}
			and so 
			\begin{align*}
				\lim\limits_{n\rightarrow \infty}\int_{\mathbb{R}^2}\vert v_n\vert^p \dd x\dd y=0.
			\end{align*}
			Then using Lemma \ref{lm:estimate_energy}-(iii), we have 
			\begin{align*}
				\lim _{n \rightarrow \infty} \int_{\mathbb{R}^2} F\left(v_n\right) \dd x\dd y=\lim _{n \rightarrow \infty} \int_{\mathbb{R}^2} t^{-s-1}F\left(t^{\frac{s+1}{2}}v_n\right) \dd x\dd y=0.\quad \text{for all } t> 0.
			\end{align*}
% \red{			Set $\tilde{v}_n(x,y):=t_n^{-\frac{s+1}{2}}v_n(t_n^{-s} x, t_n^{-1} y)$ and $Q(\tilde{v_n})=Q(u_n)=0$.}
Observe that $u_n=(v_n)_{t_n^{-1}}$. Therefore, by Lemma \ref{lm:maximum_point}(ii), for any $t>0$
			 we have
% \red{			\begin{align*}
% 				c\geq E(u_n)=E(\tilde{v}_n)\geq E(v_{n,t})=\frac{t^{2s}}{2}-\int_{\mathbb{R}^2} t^{-s-1}F\left(t^{\frac{s+1}{2}}v_n\right) \dd x\dd y=\frac{t^{2s}}{2}+o_n(1),
% 			\end{align*}}
			\begin{align*}
				c\geq E(u_n)=E\left((v_n)_{t_n^{-1}}\right)\geq E\big((v_n)_{t}\big)=\frac{t^{2s}}{2}-\int_{\mathbb{R}^2} t^{-s-1}F\left(t^{\frac{s+1}{2}}v_n\right) \dd x\dd y=\frac{t^{2s}}{2}+o_n(1),
			\end{align*}
			which is a contradiction via choosing $t>0$ large enough.
		\end{enumerate}
	\end{proof}
	
	\subsection{Ground States}
	\begin{defi}\label{def:homotopy} \cite[Definition 3.1]{Gh}
		Let $B$ be a closed subset of a set $Y \subset H$. We say that a class $\mathcal{G}$ of compact subsets of $Y$ is a homotopy stable family with the closed boundary $B$ provided that
		\begin{enumerate}
			\item [\textnormal{(i)}] every set in $\mathcal{G}$ contains $B$;
			\item [\textnormal{(ii)}] for any $A \in \mathcal{G}$ and any function $\eta \in C([0, 1] \times Y, Y)$ satisfying $\eta(t, x)=x$ for all $(t, x) \in (\{0\} \times Y) \cup([0, 1] \times B)$, then $\eta(\{1\} \times A) \in \mathcal{G}$.
		\end{enumerate}
	\end{defi}
	\begin{lem}\label{lm:ps}
		Let  $\mathcal{G}$ be a homotopy stable family of compact subsets of $S_c$ with closed boundary $B$ and set
		\begin{align} \label{eq:ming}
			\gamma_{\mathcal{G}}(c):=\inf_{A\in \mathcal{G}}\max_{u \in A} \Psi(u),
		\end{align}
		where 
		$$
		\Psi(u):=E(u_{t_u})=\max_{t>0} E(u_t).
		$$
		If $\gamma_{\mathcal{G}}(c)>0$, then there exists a Palais-Smale sequence $\{u_n\} \subset P_c$, for $E$ restricted on $S_c$ at the level $\gamma_{\mathcal{G}}(c)$ for any $c>0$.
	\end{lem}
	\begin{proof}
		Let us first define a mapping $\eta: [0,1] \times S_c \to S_c$ by $\eta(s, u)=u_{1-s+st_u}$. From Lemma \ref{lm:maximum_point}, we know that $t_u=1$, if $u \in P_c$.
		Note that $B \subset P_c$, then $\eta(s ,u)=u$ for any $(s, u) \in (\{0\} \times S_c) \cup([0, 1] \times B)$. 
		Furthermore, by Lemma \ref{lm:maximum_point}, we see that $\eta \in C([0,1] \times S_c, S_c)$. Let $\{D_n\} \subset \mathcal{G}$ be a minimizing sequence to \eqref{eq:ming}. 
		In view of Definition \ref{def:homotopy}, we then get that
		$$
		A_n:=\eta(\{1\} \times D_n)=\{u_{t_u} : u \in D_n\} \in \mathcal{G}.
		$$
		Since $A_n \subset P_c$, then 
		$$
		\displaystyle\max_{v \in A_n}\Psi(v)=\displaystyle\max_{u \in D_n}\Psi(u).
		$$
		Therefore, there exists another minimizing sequence $\{A_n\} \subset P_c$ to \eqref{eq:ming}. Using \cite[Theorem 3.2]{Gh}, we then deduce that there exists a Palais-Smale sequence $\{\tilde{u}_n\} \subset S_c$ for $\Psi$ at the level $\gamma_{\mathcal{G}}(c)$ such that 
		\begin{align} \label{eq:dist}
			\mbox{dist}_{H}(\tilde{u}_n, A_n)=o_n(1).
		\end{align}
		For simplicity, we shall write $t_n=t_{\tilde{u}_n}$ and $u_n=(\tilde{u}_n)_{t_n}$ in what follows.
		
		Let us now prove that there exists a constant $C>0$ such that $t_n^{-1}\le C$, for any $n\ge 1$. Notice first that
	$$
		t_n^{-2s}=\frac{\|\partial_x \tilde{u}_n\|_2^2   + \|D_y^{s} \tilde{u}_n\|_2^2}{\|\partial_x u_n\|^2_2 + \|D_y^{s} u_n\|_2^2}.
		$$
	% \red{	$$
	% 	t_n^{2s}=\frac{\int_{\R^2}|\partial_x u_n|^2  \dd x\dd y + \int_{\R^2} |D_y^{s} u_n|^2 \dd x\dd y}{\int_{\R^2}|\partial_x \tilde{u}_n|^2  \dd x\dd y + \int_{\R^2} |D_y^{s} \tilde{u}_n|^2 \dd x\dd y}.
	% 	$$}
		Since $E(u_n)=\Psi(\tilde{u}_n)=\gamma_{\mathcal{G}}(c)+o_n(1)$, $\{u_n\} \subset P_c$ and $0<\gamma(c)<\infty$, from Lemma \ref{lm:coercive}, then  there exists a constant $C_1>0$ such that $1/C_1 \leq \|\partial_x u_n\|_{2}+\|D_y^s u_n\|_{2}\leq C_1$. 
		On the other hand, since $\{A_n\} \subset P_c$ is a minimizing sequence to \eqref{eq:ming}, from Lemma \ref{lm:coercive}, then $\{A_n\}$ is bounded in $H$. In light of \eqref{eq:dist}, we then get that $\{\tilde{u}_n\}$ is bounded in $H$. 
	 
		Consequently, the desired conclusion follows.
		
		Next we show that $\{u_n\} \subset P_c$ is a Palais-Smale sequence for $E$ restricted on $S_c$ at the level $\gamma_{\mathcal{G}}(c)$. In the following, we shall denote by $\|\cdot\|_{*}$ the dual norm of $(T_u S_c)^*$. Observe that
		\begin{align} \label{eq:norm1}
			\|dE(u_n)\|_*=\sup_{\psi \in T_{u_n}S_c, \|\psi\| \leq 1}|dE(u_n)[\psi]|=\sup_{\psi \in T_{u_n}S_c, \|\psi\|\leq 1}|dE(u_n)[(\psi_{t_n^{-1}})_{t_n}]|.
		\end{align}
		By straightforward calculations, we are able to infer that the mapping $T_uS_c \to T_{u_{t_u}}S_c$ defined by $\psi \mapsto \psi_{t_u}$ is an isomorphism. Moreover, arguing as in \cite{jeanjeanlu2}, we have that $d\Psi(u)[\psi]=dE(u_{t_u})[\psi_{t_u}]$, for any $u \in S_c$ and $\psi \in T_uS_c$. As a consequence, making use of \eqref{eq:norm1} and the fact that  $t_n^{-1}\le C$, for any $n\ge 1$, we get that
		$$
		\|dE(u_n)\|_*=\sup_{\psi \in T_{u_n}S_c, \|\psi\|\leq 1}\big|d\Psi(\tilde{u}_n)\big[\psi_{t_n^{-1}}\big]\big|
\le {t_n^{-1}}\|d\Psi(\tilde{u}_n)\|_*
\le C\|d\Psi(\tilde{u}_n)\|_*.
		$$
		Since $\{\tilde{u}_n\} \subset S_c$ is a Palais-Smale sequence for $\Psi$ at the level $\gamma_{\mathcal{G}}(c)$, we then deduce that $\{u_n\} \subset P_c$ is a Palais-Smale sequence for $E$ restricted on $S_c$ at the level $\gamma_{\mathcal{G}}(c)$. Thus the proof is completed.
	\end{proof}
	\begin{lem}
		There exists a Palais-Smale sequence $\left\{u_n\right\} \subset P_c$ for the constrained functional $E_{\mid S_c}$ at the level $\gamma(c)$.
	\end{lem}
	\begin{proof}
			Let $B=\emptyset$ and $\mathcal{G}$ be all singletons in $S_c$. In this situation, by \eqref{eq:ming}, there holds that
			$$
			\gamma_{\mathcal{G}}(c)=\inf_{u \in S_c} \sup_{t>0} E(u_t).
			$$
			Next we are going to prove that $\gamma_{\mathcal{G}}(c)=\gamma(c)$. From Lemma \ref{lm:maximum_point}, we know that, for any $u \in S_c$, there exists a unique $t_u>0$ such that $u_{t_u} \in P_c$ and $E(u_{t_u})=\max_{t >0}E(u_t)$. This then implies that
			$$
			\inf_{u \in S_c} \sup_{t>0} E(u_t)  \geq \inf_{u \in P_c} E(u).
			$$
			On the other hand, for any $u \in P_c$, we have that $E(u)=\max_{t >0}E(u_t)$. This  gives that
			$$
			\inf_{u \in S_c} \sup_{t>0} E(u_t)  \leq \inf_{u \in P_c} E(u).
			$$
			Accordingly, we derive that $\gamma_{\mathcal{G}}(c)=\gamma(c)$. So it follows from Lemma \ref{lm:ps} that the result of this lemma holds true and the proof is completed.
	\end{proof}

\begin{lem}\label{cresc}
Suppose that $(H0)-(H4)$ hold. Then the function $c\mapsto \gamma(c)$ is nonincreasing on $(0,\infty)$.
\end{lem}
\begin{proof}
We will show that for any $0<c_1<c_2$ and for any $\eps>0$ we have
\begin{equation}\label{cce}
\gamma(c_2)\le \gamma(c_1)+\eps.
\end{equation}
By definition of $\gamma(c_1)$, there exists $u\in P_{c_1}$ such that
\begin{equation}\label{cce1}
E(u)\le \gamma(c_1)+\frac{\eps}{2}.
\end{equation}
Let $\chi \in C^{\infty}(\mathbb{R}^2)$ be radial and such that
$$
\chi(x)=\left\{\begin{array}{ll}
1, & |x| \leq 1, \\
\in[0,1], & |x| \in(1,2), \\
0, & |x| \geq 2 .
\end{array}\right.
$$
For any small $\delta>0$, we define $u_{[\delta]}(x)=u(x) \cdot \chi(\delta x) \in  H\backslash\{0\}$. Since $u_{[\delta]} \rightarrow u$ in $H$ as $\delta \rightarrow 0^{+}$,
 by Lemma \ref{lm:maximum_point}-(iii), one has $t_{u_{[\delta]}}\to t_u=1$, as $\delta \rightarrow 0^{+}$, and thus
$$
( u_{[\delta]})_{t_{u_{[\delta]}}} \rightarrow u_{t_u}=u, \quad \text { in } H, \quad \text { as } \delta \rightarrow 0^{+} .
$$
So we can fix a $\delta>0$ small enough such that
\begin{equation}\label{cce2}
E\big(( u_{[\delta]})_{t_{u_{[\delta]}}} \big) \leq E(u)+\frac{\epsilon}{4}.
\end{equation}
Let $v \in C^{\infty}(\mathbb{R}^2)$ be such that $\operatorname{supp}(v) \subset B(0,1+4 / \delta) \backslash B(0,4 / \delta)$ and set
$$
\tilde{v}=\frac{c_2-\left\|u_{[\delta]}\right\|_{2}^2}{\|v\|_{2}^2} v .
$$
For any $\lambda\ge 0$, we define $w_{[\lambda]}=u_{[\delta]}+\tilde{v}_\lambda$. Since
$$
\operatorname{supp}\left(u_{[\delta]}\right) \cap \operatorname{supp}(  \tilde{v}_\lambda)=\emptyset,
$$
we have that  $w_{[\lambda]} \in S_{c_2}$. 
Since $E\Big( (w_{[\lambda]})_{t_{w_{[\lambda]}}}\Big) \geq 0$  and  $w_{[\lambda]} \rightarrow u_{[\delta]} \neq 0$ almost everywhere in $\mathbb{R}^2$, as $\lambda \rightarrow 0$, arguing as in the proof of Lemma \ref{lm:maximum_point}-(iii), we deduce that 
$t_{w_{[\lambda]}}$ is bounded from above when $\lambda \rightarrow0$. 
 Now, since
$$
t_{w_{[\lambda]}}\cdot \lambda \rightarrow 0,\quad \text { as } \lambda \rightarrow 0,
$$
we have
$$
\|\partial_x\tilde{v}_{t_{w_{[\lambda]}}\cdot\lambda}\|_2^2+
\|D_y^s\tilde{v}_{t_{w_{[\lambda]}}\cdot\lambda}\|_2^2\rightarrow 0 \quad \text { and } \quad \|\tilde{v}_{t_{w_{[\lambda]}}\cdot\lambda}\|_{2^\sharp} \rightarrow 0 .
$$
So, by Lemma \ref{lm:estimate_energy}-(iii), we have
\begin{equation}\label{cce3}
E\left(\tilde{v}_{t_{w_{[\lambda]}}\cdot\lambda}\right) \leq \frac{\varepsilon}{4}, \quad \text { for } \lambda>0 \text { small enough} .
\end{equation}
Therefore, combining \eqref{cce1}, \eqref{cce2}, and \eqref{cce3}
$$
\begin{aligned}
\gamma(c_2) \leq E\Big((w_{[\lambda]})_{t_{w_{[\lambda]}}}\Big) & =E\left((u_{[\delta]})_{t_{w_{[\lambda]}}}\right) 
+E\left(\tilde{v}_{t_{w_{[\lambda]}}\cdot \lambda}\right) \\
& =E\left((u_{[\delta]})_{t_{u_{[\delta]}}}\right) 
+E\left(\tilde{v}_{t_{w_{[\lambda]}}\cdot \lambda}\right) \\
& \leq E(u)+\frac{\varepsilon}{2} \leq \gamma(c_1)+\varepsilon,
\end{aligned}
$$
that is \eqref{cce}.
\end{proof}

	\begin{lem}
		Let $\left\{u_n\right\} \subset S_c$ be any bounded Palais-Smale sequence for the constrained functional $E_{\mid S_c}$ at the level $\gamma(c)>0$, satisfying $Q\left(u_n\right) \rightarrow 0$.
		Then there exists $u \in S_c$ and $\mu>0$ such that, up to the extraction of a subsequence and $u p$ to translations in $\mathbb{R}^2, u_n \rightarrow u$ strongly in $H$ and $	-\partial_{xx} u + D_y^{2s} u + \alpha u=f(u)$.
	\end{lem}
	\begin{proof}
By Lemma \ref{concentration} and Lemma \ref{lm:estimate_energy}, since $E(u_n)\to \gamma(c)>0$, we have that there exist a sequence $\{y_n\} \subset \R^2$ and a nontrivial function $u \in H$ such that $u_n(\cdot-y_n) \rightharpoonup u$ weakly in $H$, as $n \to \infty$. Since $\{u_n\} \subset S_c$ is a Palais-Smale sequence for $E$ restricted on $S_c$, then $u_n \in H$ satisfies the equation
		\begin{align} \label{equ21}
			-\partial_{xx} u_n + D_y^{2s} u_n + \alpha_n u_n=f(u_n) +o_n(1),
		\end{align}
		where
		$$
		\alpha_n=\frac 1 c \left (\int_{\R^2}F(u_n) \dd x\dd y -\|\partial_ x u_n\|^2_2 -\|D_y^{s} u_n\|_2^2  \right) +o_n(1).
		$$
		Note that $\{u_n\} \subset S_c$ is bounded in $H$, then there exists a constant $\alpha \in \R$ such that $\alpha_n \to \alpha$ in $\R$, as $n \to \infty$. It then follows from \eqref{equ21} that $u \in H$ satisfies the equation
		\begin{align} \label{equ22}
			-\partial_{xx} u + D_y^{2s} u + \alpha u=f(u).
		\end{align}
		According to Theorem \ref{ph}, then $Q(u)=0$. Therefore we have
\[
\alpha \|u\|_2^2
=\int_{\R^2}\left(f(u)u-\frac{s+1}{2s}\tilde{F}(u)\right) \dd x\dd y
=\int_{\R^2}\left(\frac{s-1}{2s}f(u)u+\frac{s+1}{s}F(u)\right) \dd x\dd y,
\]
and so, by (H5),  we get that $\alpha>0$.
		Define $v_n:=u_n(\cdot-y_n)-u$. From Brezis-Lieb's lemma, see also \cite[Lemma 2.6]{jeanjeanlu2}, we can conclude that
		\begin{align} \label{bl1}
			0=Q(u_n)=Q(v_n)+Q(u)+o_n(1)=Q(v_n)+o_n(1)
		\end{align}
		and
		\begin{align}\label{bl2}
			\gamma(c)=E(u_n)+o_n(1)=E(v_n)+E(u)+o_n(1)\ge E(v_n)+\gamma(\|u\|_2^2) +o_n(1).
		\end{align}
		Since, by  Brezis-Lieb's Lemma, $\|u\|_2^2 \leq c$, then $\gamma(c) \leq \gamma(\|u\|_2^2)$, see Lemma \ref{cresc}. As a result, \eqref{bl2} leads to $E(v_n) \leq o_n(1)$. Furthermore, according to \eqref{bl1}, there holds that $Q(v_n)=o_n(1)$. Therefore, we have that $E(v_n)=o_n(1)$ and, by Lemma \ref{lm:nonlinearity}
		$$
		\|\partial_ x v_n\|_2^2  +\|D_y^{s} v_n\|_2^2  =o_n(1), \quad \int_{\R^2} F(v_n)  \dd x\dd y=o_n(1), \quad  \int_{\R^2} f(v_n)v_n  \dd x\dd y=o_n(1).
		$$
		Thanks to $\alpha>0$, from \eqref{equ21} and \eqref{equ22}, we then get that $\|v_n\|_2=o_n(1)$. Consequently,  we conclude that $u_n(\cdot-y_n) \to u$ in $H$, as $n \to \infty$. This completes the proof.
	\end{proof}
	Using above lemma, we can prove the following theorem.
	\begin{thm}\label{groun-thm}
		Suppose that $(H0)-(H5)$ hold. Then, for any $c>0$, there exists a normalized ground state whose associated Lagrange multiplier $\alpha$ is positive.
	\end{thm}

\section{Blow-up criteria} \label{blowup-section}
Now, in this section, we study the Cauchy problem \eqref{eq0}. We will find the conditions under which the local solution blows up.
The arguments are inspired by \cite{BHL}, but for our problem, the natural absence of radial symmetry necessitates considering more regular solutions.

In the following we denote by  $\langle f, g\rangle = \int_{\R^2} \bar f g \dd x\dd y$ the inner product on $L^2(\R^2)$ and $[X, Y] = XY- Y X$ denotes the commutator of $X$ and $Y$.

We assume throughout this section that  $s\in (1/2,1)$, $\frac{2(3s+1)}{1+s}<p<\frac{2(1+s)}{1-s}$,  and \eqref{eq0} is well-posed in $H$. However, the following local well-posedness was reported in \cite{choiaceves}. 
\begin{prop}

Let $p\geq4$ and $s\neq 1/2$.
 
Suppose that $s_1\in(s_c,[p]-2]$ when $p\notin\Z_{\rm odd}$, and
$s_1\in(s_c,\infty)$ for $p\in\Z_{\rm odd}$, where $$s_c=\frac12+\frac{1}{2s}-\frac{2}{p-2}.$$ 
 
Then, the Cauchy problem associated with \eqref{eq0} is locally well-posed in $H^{(s_1,s_1s)}$.
\end{prop}

Let $u(x,y,t)\in C([0,T);H)$ be a solution of \eqref{eq0}. We define the localized virial identity  of $u$ by
\[
V_\phi(u(t))=\int_\R\phi(x,y)|u(x,y,t)|^2\dd x\dd y,
\]
where $\phi:\R^2\to\R$ is bounded.

\begin{lem}\label{well-defn}
If $\phi_x\in L^\infty$ and $\phi_y\in L^\infty_xW^{2,\infty}_y$, then
$
\frac{\dd}{\dd t}V_\phi(u(t)) 
$
is well-defined, for any $t\in[0,T)$.
\end{lem}
\begin{proof}
	Without loss of generality, we assume that $u\in C_0^\infty(\R^2)$. Define $$u_m=\sqrt{\frac{\sin \pi s}{\pi}}(m-\partial_{yy})^{-1}u(t),$$ then
	\[
\begin{split}
		\frac{\dd}{\dd t}V_\phi(u(t))
	&=i\scal{u,[-\partial_{xx}+D_y^{2s},\phi]u}
	\\&
	=i\scal{u,[-\partial_{xx} ,\phi]u}
	+i\scal{u,[ D_y^{2s},\phi]u}
			\\&
		= -\ii\scal{\phi_{xx},|u|^2}- 2i\scal{u_x,\phi_x u}
-i		\int_0^{+\infty} m^s\int_{\R^2}\paar{\phi_{yy}|u_m|^2+2\bar{u}_m\phi_y(u_m)_y}\dd x\dd y\dd m.
			 \end{split}
	\]
In the above inequalities, we have used the  following Balakrishnan’s formula	
\[
[D_y^{2s},B]
=\frac{\sin(\pi s)}{\pi}\int_0^{+\infty} m^s\frac{1}{m-\partial_{yy}}[-\partial_{yy},B]\frac{1}{m-\partial_{yy}}\dd m.
\]

It is clear from the Cauchy-Schwartz inequality that
\[
|\scal{\phi_{xx},|u|^2}+\scal{u_x,\phi_x u}|
\lesssim \|\phi_x\|_{L^\infty}\|u_x\|_{L^2}\|u\|_{L^2}.
\]
On the other hand, by an argument, similar to \cite[Lemma A.2]{BHL}, one can show that
\begin{equation*}%\label{est-bl-1}
\left|\int_0^{+\infty} m^s\int_{\R^2}\bar{u}_m\phi_y(u_m)_y\dd x\dd y\dd m\right|\lesssim 
\norm{\phi_y}_{L_x^\infty W^{1,\infty}_y}\norm{u}_H^2 
 \end{equation*}
 and
\begin{equation*}%\label{est-bl-2}
\left|\int_0^\infty m^s\int_{\R^2}\phi_{yy}|u_m|^2\dd x\dd y\dd m\right|\lesssim\|\phi_{yy}\|_{L^\infty}^{2s-1}
\|\phi_{y}\|_{L^\infty}^{2(1-s)}\|u\|_2^2.
 \end{equation*}
  Then the boundedness of $\frac{\dd}{\dd t}V_\phi(u(t))$ follows from
\begin{equation}\label{bound-V}
	\begin{split}
\abso{\frac{\dd}{\dd t}V_\phi(u(t))}
\lesssim
\paar{\|\phi_y\|_{L_x^\infty W^{1,\infty}_y}
+
\|\phi_x\|_{L^\infty}}\|u(t)\|_H^2.
	\end{split}
\end{equation} 
\end{proof}

Associated with $V_\phi$, now we define 
\[
M_\phi(u(t))=2\Im \scal{u,\nabla\phi\cdot \nabla u}.
\]

 \begin{lem}\label{lem-dM}
 	Let $\phi(x,y)=\ff(x)+\psi(y)$ and $\ff,\psi\in W^{4,\infty}$. Then for any solution $u\in C([0,T); H)$ of \eqref{eq0}, we have, for any $t\in[0,T)$, that
\[
\begin{split}
	\frac{\dd}{\dd t}M_\phi(u(t))
&=
\int_{\R^2}\paar{4\ff''|u_x|^2-\ff''''|u|^2}\dd x\dd y
+\int_0^{+\infty} m^s\int_{\R^2}\paar{
	4\psi''|(u_m)_y|^2-\psi''''|u_m|^2
}\dd x\dd y\dd m\\&\qquad
- \frac{ 2(p-2)}{p}\int_{\R^2}\Delta\phi|u|^p\dd x\dd y.
\end{split}
\]
 \end{lem}
\begin{proof}
	Let $\Gamma_\ff(u)=-i(2\ff'u_x+\ff''u) $ and $\Gamma_\psi(u)=-i(2\psi'u_y+\psi''u)$. Then it can be readily seen that
	\[
	M_\phi(u)=\scal{u,\Gamma_\ff u}+
	\scal{u,\Gamma_\psi u}
	\]
	and
	\[
\begin{split}
		\frac{\dd}{\dd t}M_\phi (u(t))&=
	\scal{u,[-\partial_{xx},i\Gamma_\ff] u}
	+\scal{u,[D_y^{2s},i\Gamma_\ff] u}+
		\scal{u,[-\partial_{xx},i\Gamma_\psi] u}
	+\scal{u,[D_y^{2s},i\Gamma_\psi] u}
 \\
 &\qquad -\scal{u,[|u|^{p-2},i\Gamma_\ff] u} 
	-\scal{u,[|u|^{p-2},i\Gamma_\psi] u}
	\\
	&
	=\scal{u,[-\partial_{xx},i\Gamma_\ff] u}- \scal{u,[|u|^{p-2},i\Gamma_\ff] u} 
	+\scal{u,[D_y^{2s},i\Gamma_\psi] u}-\scal{u,[|u|^{p-2},i\Gamma_\psi] u}.
	\end{split}
	\]
		A direct calculation shows that
	\[
	\scal{u,[-\partial_{xx},i\Gamma_\ff] u}-\scal{u,[|u|^{p-2},i\Gamma_\ff] u}
	=
	4\scal{\ff'',|u_x|^2}-\scal{\ff'''',|u|^2}-\frac{2(p-2)}{p}
	\int_{\R^2}\ff''|u|^p\dd x\dd y.
	\]
	On the other hand, it follows from Lemma 2.1 in \cite{BHL} that
\[
\begin{split}
	\scal{u,[D_y^{2s},i\Gamma_\psi] u}
 -\scal{u,[|u|^{p-2},i\Gamma_\psi] u}
	&=
	\int_0^{+\infty}
	m^s \int_{\R^2}\paar{
		4\psi''|(u_m)_y|^2-\psi''''|u_m(t)|^2
	}\dd x\dd y\dd m
 \\
	&\qquad-
	\frac{2(p-2)}{p}
	\int_{\R^2}\psi''|u|^p\dd x\dd y.   
\end{split}
		\]
\end{proof}

Let $\theta:[0,+\infty)\to[0,+\infty)$ be a smooth function such that
$\theta''(r)\leq 2$, $\theta(r)=r^2$, if $r\in[0,1)$, and $\theta(r)=2$, if $r\geq2$. Given $R>1$, define $\theta_R(r)=R^2\theta(r/R)$ .
\\
It is not hard to see, from the definition of $\theta_R$, that $\|\frac{\dd ^k}{\dd r^k}\theta_R\|_{L^\infty}\lesssim R^{2-k}$, for $k=0,\cdots, 4$, and
\[
\text{supp}\paar{\frac{\dd ^k}{\dd r^k}\theta_R}\subset\begin{cases}
	\{|r|\leq 2R\} &k=1,2,\\
		\{R\leq |r|\leq 2R\} &k=3,4.
\end{cases}
\]

Denote
\[
M_{\phi_R}(u)=2\Im\scal{u,\nabla\phi_R\cdot\nabla u},
\]
where $\phi_R(x,y)=s\theta_R(x)+\theta_R(y)$. Notice that $M_{\phi_R}$ is now well-defined for any $u\in H$. To compensate lack of symmetry, we need to assume that %$u \in C([0,T),H^+)$ with $H^+=H^{(1,s)1^+}$, implying 
	$u(t)\in C([0,T), L^q)$, for some $q>p$,   to control the nonlinear term in the following lemma.
\begin{lem}\label{lem-est}
%\torange{I don't think that it is necessary to assume the following red part here. We should need it in the next Lemma. So I suggest to delete here and to add later.}
%{\color{red}Assume that the solution $u(t)\in H$ of \eqref{eq0} exists globally in time and
% \[
% \sup_t\|u(t)\|_H<+\infty. 
% \]}
	There exists $C>0$, independent of $R$, such that
	\[
		\frac{\dd }{\dd t} M_{\phi_R}(u)
		\leq  
		8Q(u(t))+
		C\paar{ R^{-2}\|u(t)\|_{2}^2+R^{-2s}\|u(t)\|_{2}^2}
		+
		C\|u(t)\|_{ L^2 (\{|x|,|y|\geq R\})}^{q_0},
	\]
for some $q_0\in (2,p)$.
\end{lem}

\begin{proof}
By Lemma \ref{lem-dM} we know that
\begin{equation*}\label{dM}
\begin{split}
\frac{\dd}{\dd t}M_{\phi_R}(u(t))
&=
\int_{\R^2}\paar{4s\theta_R''|u_x|^2-s\theta_R''''|u|^2}\dd x\dd y
+\int_0^{+\infty} m^s\int_{\R^2}\paar{
	4\theta_R''|(u_m)_y|^2-\theta_R''''|u_m|^2
}\dd x\dd y\dd m\\&\qquad
- \frac{2(p-2)}{p}\int_{\R^2}\Delta\phi_R|u|^p\dd x\dd y
\end{split}
\end{equation*}
and we have to estimate each term of the previous estimate.\\
Being $2-\theta_R''\geq0$,
\begin{equation*}\label{dM1}
4s\int_{\R^2}\theta_R''|u_x|^2\dd x\dd y
=
8s\|u_x\|_2^2
-4s\int_{\R^2}(2-\theta_R'')|u_x|^2\dd x\dd y
\le 8s\|u_x\|_2^2.
\end{equation*}
 
Moreover 
\begin{equation*}\label{dM2}
\scal{\theta_R'''',|u|^2}\leq R^{-2}\|u(t)\|_{2}^2.
\end{equation*}
By using the identity
\[
\int_0^{+\infty} m^s \int_{\R^2}|(u_m)_y|^2\dd y\dd x \dd m
=s\|D_y^s u\|_{2}^2,
\]
and, since $2-\theta_R''\geq0$, we have
\begin{equation*}\label{dM3}
\begin{split}
	4\int_0^{+\infty} m^s\int_{\R^2}
	\theta_R''|(u_m)_y|^2\dd x\dd y\dd m
&=8s\|D_y^s u\|_{2}^2
-4\int_0^{+\infty} m^s\int_{\R^2}
(2-	\theta_R'')|(u_m)_y|^2\dd x\dd y\dd m\\&
\le 8s\|D_y^s u\|_{2}^2.
\end{split}
\end{equation*}
By \cite[Lemma A.2]{BHL},
\begin{equation*}\label{dM4}
\int_0^{+\infty} m^s\int_{\R^2}\theta_R''''(y)|u_m|^2\dd x\dd y\dd m\leq 
\|\theta_R''''\|_\infty^s\|\theta_R''\|_\infty^{1-s}\|u\|_2
\le 
R^{-2s}\|u(t)\|_2^2.
\end{equation*}
Finally
\begin{equation*}
- \frac{2(p-2)}{p}\int_{\R^2}\Delta\phi_R|u|^p\dd x\dd y
=-\frac{4(p-2)(s+1)}{p}\|u\|_{p}^p
- \frac{2(p-2)}{p}\int_{\R^2}\big(\Delta\phi_R- 2(s+1)\big)|u|^p\dd x\dd y.
\end{equation*}
It follows from $|2(s+1)-\Delta\phi_R|\lesssim 1$, and $\text{supp}(2(s+1)-\Delta\phi_R)\subset\{|x|,|y|\geq R\}$ and an interpolation that
\[
\int_{\R^2}\paar{2(s+1)-\Delta\phi_R}|u|^p\dd x\dd y
\lesssim
\|u(t)\|_{ L^2 (\{|x|,|y|\geq R\})}^{q_0},
\]
for some $2<q_0<p$.  
\\
Now the conclusion follows collecting all the previous estimates.
\end{proof}
\begin{lem}\label{stima-le2r}
Assume that the solution $u(t)\in H$ of \eqref{eq0} exists globally in time and
\[
\sup_t\|u(t)\|_H<+\infty. 
\]
Let $\epsilon > 0$ and $R > 1$. Then there exists a constant $C > 0$, independent of $R$, such that, for any $0\leq t \leq   \frac{\epsilon R}{C}$,   the following inequality holds:
\[
\norm{u(t)}_{L^2(\{|x|,|y|\geq R\})} ^{q_0}\leq o_R(1) + \epsilon.
\]
\end{lem}
\begin{proof}
Consider a smooth function $\Phi: [0,+\infty) \to [0,1]$   satisfying
	\[
	\Phi(r)= \left\{
	\begin{array}{cl}
		0 &  0 \leq r \leq \frac{1}{2}, \\
		1 &  r \geq 1.
	\end{array}
	\right.
	\]
	Given $R>1$, we denote the radial function 
	\[
	\Psi_R(x,y) = \Psi_R(r) := \Phi\paar{\frac rR}, \quad r^2=x^2+y^2. 
	\]
	It is easy to check that
		\begin{align*}
		\|\nabla \Psi_R\|_{W^{1,\infty}} \sim \|\nabla \Psi_R\|_{L^\infty} + \|\Delta \Psi_R\|_{L^\infty}  \lesssim R^{-1}. %\label{estimate psi_R}
	\end{align*}
	We next define 
	\[
	V_{\Psi_R}(u(t)) := \int_\rt \Psi_R(x,y) |u(t,x,y)|^2 \dd x\dd y.
	\]
	By the fundamental theorem of calculus, if  $u_0\in H$ is the initial datum associated to $u$, we have
	\[
	V_{\Psi_R}(u(t)) = V_{\Psi_R}(u_0) + \int_0^t \frac{\dd}{\dd\tau} V_{\Psi_R}(u(\tau)) \dd\tau \leq V_{\Psi_R}(u_0) + \left(\sup_{\tau \in [0,t]} \left|\frac{\dd}{\dd\tau} V_{\Psi_R}(u(\tau)) \right| \right) t. 
	\]
With similar arguments, we get from \eqref{bound-V} that
	\begin{align*}
		\sup_{\tau \in [0,t]} \left|\frac{\dd}{\dd\tau} V_{\Psi_R}(u(\tau)) \right|  \lesssim  \|\nabla \Psi_R\|_{W^{1,\infty}} \sup_{\tau \in [0,t]} \|u(\tau)\|^2_{H} \leq CR^{-1}.
	\end{align*}
	So that
	\[
	V_{\Psi_R}(u(t)) \leq V_{\Psi_R}(u_0) + CR^{-1} t.
	\]
	The conservation of mass gives
	\[
	V_{\Psi_R}(u_0) = \int_\rt \Psi_R(x) |u_0(x,y)|^2 \dd x\dd y   \leq \int_{\{|y|,|x|>R/2\}} |u_0(x,y)|^2 \dd x\dd y  \to0,
	\]
	as $R\rightarrow +\infty$. The proof is complete by using the fact
	\[
	\int_{\{|x|,|y|\geq R\}} |u(x,y,t)|^2 \dd x\dd y \leq V_{\Psi_R}(u(t)).
	\]     
\end{proof}

 Now we are in a position to state our main blow-up result. Due to lack of symmetry, we need to assume that the local solutions have some additional regularity and integrability. Such  assumptions were posed for the fractional Nonlinear Schr\"odinger equations in \cite{BHL}. So the auxiliary space of the following result can be replaced with $H^{(1,s_1)}$ with a sufficiently large $s_1>s$.
\begin{thm}\label{blow-thm-1}
Let   $u(t)\in C([0,T);H)\cap C([0,T); L^q)$, for some $q>p$, be a solution of \eqref{eq0} associated with the initial datum $u_0\in H$.	If there exists $\delta>0$ such that 
\begin{equation}\label{blow-condition-thm-1}
	 Q(u(t))\leq-\delta<0, 
\end{equation}
 for any $t\in[0,T)$, then  
 the solution $u(t)$ blows up in finite time, or $u(t)$ blows up at infinity and $\|u(t)\|_{L^{q}}\to+\infty$, as $t\to+\infty$, for any $q>p$.
\end{thm}

\begin{proof}
 
	We can assume that $T=+\infty$. Suppose, by contradiction, that $u(t)$ is a global solution such that $\|u(t)\|_{L^q}<+\infty$, for all $t>0$ and some $q>p$. The mass and energy conservations together with an interpolation shows that $\|u(t)\|_{H}<+\infty$, for all $t>0$. From Lemma \ref{lem-est} and Lemma \ref{stima-le2r}, for any $\epsilon>0$ and $R>1$, there is $C>0$, independent of $R$, such that
\[
\begin{split}
		\frac{\dd }{\dd t} M_{\phi_R}(u)&\leq C\paar{
	8Q(u(t))+
 \paar{R^{-2s}+R^{-2}}}+
\paar{ o_R(1)+\epsilon}^{q_0}\\
&\leq C\paar{-8\delta+
	\paar{R^{-2s}+R^{-2}}+
	\paar{ o_R(1)+\epsilon}^{q_0}},
\end{split}
\]
for any $t>0$. By choosing $\epsilon$ sufficiently small and $R$ sufficiently, we see that 
\[
\frac{\dd }{\dd t} M_{\phi_R}(u)\leq -\delta<0,
\]
for any $t>0$. This leads us to
\[
  M_{\phi_R}(u)\leq -ct,
\]
for some constant $c=c(\delta)>0$ and any $t>t_0$ with sufficiently large $t_0\ge 0$. By a standard argument, we can now obtain from \eqref{bound-V} that
\[
\|u(t)\|_{\dot{H}}\gtrsim t^a,
\]
for some $a\in(0,1)$ and any $t\ge t_0$. This shows that
$\|u(t)\|_{\dot{H}}\to+\infty$, as $t\to +\infty$,  reaching a contradiction.

\end{proof}

\begin{thm}\label{bl-c-thmm}
 Let $s\in(1/2,1)$ and $\frac{2(3s+1)}{1+s}<p<\frac{2(1+s)}{1-s}$. 
 
 Assume that $u\in C([0,T);H)\cap C([0,T); L^q)$, for some $q>p$,
 
 is a solution of \eqref{eq0} such that $u(0)=u_0\in H$. Furthermore, suppose that $E(u_0)<0$, or   $E(u_0)\geq0$ such that
 
 \begin{equation}\label{blow-cnd}
\begin{split}
			 &\frac{E(u_0)}{E (\ff)}<
         \paar{\frac{s^s }{(s+1)^{s+1}}}^{\frac{p-2}{(p-2)(1+s)-4s}}
    \|\ff\|_\lt^{2\rho},\\
			 &
 	 		 \frac{\|u_0\|_{\dot{H}}^2}{\|\ff\|_{\dot{H}}^2}  >
           \paar{\frac{s^s }{(s+1)^{s+1}}}^{\frac{p-2}{(p-2)(1+s)-4s}}
	 \|\ff\|_\lt^{2\rho},
\end{split}
	\end{equation} 
	where $\ff$ is a ground state of \eqref{equ-0} and
	$\rho=\frac{ p(s-1)+2(s+1) }{(p-2)(1+s)-4s}$.
	Then  $u(t)$ blows up in finite time, or
   there exists a time sequence $\{t_n\}$ such that $t_n\to+\infty$ and
   \[
   \lim_{n\to+\infty}\|u(t_n)\|_{\dot{H}}=+\infty.
   \]
In the case $p=\frac{2(3s+1)}{1+s}$, the solution $u(t)$ blows up in finite or infinite time provided $E(u_0)<0$.

\end{thm}
\begin{proof}
	It suffices to check, from Theorem \ref{blow-thm-1}, that \eqref{blow-condition-thm-1} holds.

 Let us start with the case $E(u_0)<0$. Hence, the energy conservation reveals that
	\[
	Q(u(t))
	=2sE(u(t))-\frac{(s+1)(p-2)-4s}{2p}\|u(t)\|_{L^p}^p
	\leq 2sE(u(t))
	\]
	for all $t>0$. Hence, \eqref{blow-condition-thm-1} follows with $\delta=-2sE(u(t))$.

	Next, we consider the case $E(u_0)\geq0$ and so the mass supercritical case.
	
	By using Lemma \ref{gn-lemma}, we have    that
	% \[
	% \begin{split}
	% 	E(u(t))\|u(t)\|_{L^2}^{2\rho}  
	% 	&=\frac12\|u(t)\|_{\dot{H}}^2
	% 	\|u(t)\|_{L^2}^{2\rho}-
	% 	\frac1p\|u(t)\|_{L^p}^p \|u(t)\|_{L^2}^{2\rho}\\&
	% 	\geq 
	% 	\frac12\left(\|u(t)\|_{\dot{H}}\|u(t)\|_{L^2}^{ \rho}\right)^2
	% 	-
	% 	\frac{C_{p,s}}{p}
	% 	\paar{\|u(t)\|_{\dot{H}}\|u(t)\|_{L^2}^{\rho}}^\frac{(p-2)(1+s)}{2s}
	% 	\\&
	% 	=:g\paar{\|u(t)\|_{\dot{H}}\|u(t)\|_{L^2}^{\rho}},
	% \end{split}
	% \]
 	\[
	\begin{split}
		E(u(t))  
		&=\frac12\|u(t)\|_{\dot{H}}^2
		 -
		\frac1p\|u(t)\|_{L^p}^p  \\&
		\geq 
		\frac12 \|u(t)\|_{\dot{H}}^2
		-
		\frac{C_{p,s}}{p}
		 \|u(t)\|_{\dot{H}} ^\frac{(p-2)(1+s)}{2s}
				=:g\paar{\|u(t)\|_{\dot{H}} },
	\end{split}
	\]
	where $g(X)=\frac12X^2-\frac{C_{p,s}}{p}X^\frac{(p-2)(1+s)}{2s}$. It is easily seen that $g$ is increasing on $(0, X_0)$ and decreasing on $(X_0,+\infty)$, where
 
  \[X_0^2=
\frac{(p-2) s^{\frac{s(p-2)}{(p-2)(1+s)-4s}}}{(p(s-1)+2(1+s)) (1+s)^{\frac{4s}{(p-2)(1+s)-4s}}}
 \|\ff\|_\lt^\frac{4s(p-2)}{(p-2)(1+s)-4s}
	=
 \frac{s^{\frac{s(p-2)}{(p-2)(1+s)-4s}}}{{(s+1)}^{\frac{(p-2)(1+s)}{(p-2)(1+s)-4s}}}
	\paar{\|\ff\|_\lt^\rho\|\ff\|_{\dot{H}}}^2.\] 
	 To compute $X_0$, we have used \eqref{best-con} and the following identities (see \cite{E})
	 \begin{equation*}%\label{eq-p}
	 	\begin{split}
	 		s\|D^s_y\ff\|_\lt^2&= \|\ff_x\|_\lt^2=\frac{s(p-2)}{p(s-1)+2(1+s)}\|\ff\|_\lt^2\\
	 		 	\|\ff\|_{L^p}^p&= \frac{2ps}{p(s-1)+2(1+s)}\|\ff\|_\lt^2.
	 		 	 	\end{split}
	 \end{equation*}
Hence, we obtain that
  \[\begin{split}
 2 E(\ff)&=\frac{p(1+s)-6s-2}{ p(s-1)+2(1+s) }\|\ff\|_\lt^2
 ,\qquad
 \|\ff\|_{\dot H}^2 =\frac{(s+1)(p-2)}{p(s-1)+2(1+s)}\|\ff\|_\lt^2,\\
 \|\ff\|^2_{\dot H}&=\frac{2(p-2)(1+s)}{p(1+s)-6s-2}E(\ff),
 \end{split}
  \]
  and
 
  	 \[
	 g(X_0)=\frac{(p-2)(1+s)-4s}{2(p-2)(1+s)}X_0^2=
 \frac{s^{\frac{s(p-2)}{(p-2)(1+s)-4s}}}{{(s+1)}^{\frac{(p-2)(1+s)}{(p-2)(1+s)-4s}}}
	 \|\ff\|_\lt^{2\rho}E(\ff)
  <\|\ff\|_\lt^{2\rho}E(\ff).
	 \] 
 	 Now the conservation of mass and energy  imply, together with \eqref{blow-cnd}, that
	 \[
	g \paar{\|u(t)\|_{\dot{H}} }
	 \leq 
 E(u(t))
	 =
 E(u_0)
	 <
	 g \paar{X_0},
	 \]
	 for all $t\in [0,T)$. Condition \eqref{blow-cnd} shows, together with the continuity argument, that
	 \[
	 \|u(t)\|_{\dot{H}} >X_0,
	 \]
	 for any $t\in [0,T)$. On the other hand, we have, again from \eqref{blow-cnd}, that there exists $\delta_0\in(0,1)$ such that
	 \[
	 E(u_0) <(1-\delta_0)g(X_0).
	 \]
	 Consequently, we get
	 \[
	 \begin{split}
	 	Q(u(t)) 
	 	&=
	 	 \frac{(s+1)(p-2)}{2}E(u(t))-\frac{(s+1)(p-2)-4s}{4}\|u(t)\|_{\dot{H}}^2 \\
	 	&
	 	\leq \paar{ \frac{s^s }{(1+s)^{1+s}}}^{\frac{p-2}{(p-2)(1+s)-4s}} \|\ff\|_\lt^{2\rho}
	 	\paar{\frac{(s+1)(p-2)}{2}(1-\delta_0)  
	 E(\ff)
	 	-
\frac{(s+1)(p-2)-4s}{4}
	\|\ff\|_{\dot H}^2}\\
	 	&=-\frac{(1+s)(p-2)}{2} \delta_0E(\ff)\|\ff\|_\lt^{2\rho},
	 \end{split}
	 \]
	 for all $t\in [0,T)$. This implies that \eqref{blow-condition-thm-1} holds with with $\delta=\frac{(1+s)(p-2)}{2} \delta_0E(\ff)\|\ff\|_\lt^{2\rho}$. This completes the proof. 
\end{proof}
\begin{rem}
Following the proof of Theorem \ref{bl-c-thmm}, one can observe, when
  $\frac{2(3s+1)}{1+s}<p<\frac{2(1+s)}{1-s}$ and  $E(u_0)\geq0$,   that  
  $u(t)$ is uniformly bounded  in time in $H$, provided
\begin{equation}\label{blow-cnd-11}
\begin{split}
			 &\frac{E(u_0)}{E (\ff)}<
    \paar{\frac{s^s }{(s+1)^{s+1}}}^{\frac{p-2}{(p-2)(1+s)-4s}}
    \|\ff\|_\lt^{2\rho},\\
			 &
 	 		 \frac{\|u_0\|_{\dot{H}}^2}{\|\ff\|_{\dot{H}}^2}  <\paar{\frac{s^s}{(s+1)^{s+1}}}^{\frac{p-2}{(p-2)(1+s)-4s}}
	 \|\ff\|_\lt^{2\rho}.
\end{split}
	\end{equation}
 Hence, \eqref{blow-cnd} and \eqref{blow-cnd-11} provide   a sharp dichotomy between the global and blow-up solutions.
\end{rem}

As a consequence of the previous blow-up results, We conclude this section proving the instability of standing waves.

From Section \ref{supercritical-section}, we set
	\begin{align} \label{smin}
			\gamma(c):=\inf_{u \in P_c} E(u),
		\end{align}
	where
	$$
	P_c:=\{u \in S_c : Q(u)=0\}.
	$$
Let $\ff$ be a minimizer of \eqref{smin}.

\begin{thm}\label{unstable}
 
 Let $\frac{2(3s+1)}{1+s}<p<\frac{2(1+s)}{1-s}$, $s>1/2$ and $c>0$. Let $\ff$ be any minimizer of \eqref{smin} and $\alpha_c$ the corresponding Lagrangian multiplier. Then the standing wave
	$\ee^{\ii\alpha_c t}\ff(x,y)$ is strongly unstable; that is, there exists $\{u_{0,n}\}\subset H$, with $u_{0,n}\to \ff$ in $H$, as $n \to +\infty$, such that the solutions $u_n(t)$ of \eqref{eq0}, with $u_n(0)=u_{0,n}$, blow up in finite or infinite time, for any $n\geq1$.
\end{thm}

\begin{proof}
 Define  $\ff_\lambda(x,y)=\lambda^{\frac{s+1}{2}}\ff(\lambda^sx,\lambda y)$. A direct calculation shows, from the fact $Q(\ff)=0$, that
	$\partial_\lambda E(\ff_\lambda)>0$, for all $\lambda\in(0,1)$, and $\partial_\lambda E(\ff_\lambda)<0$, if $\lambda\in(1,+\infty)$. This implies that $E(\ff_\lambda)<E(\ff)$, for $\lambda>1$. Take $\{\lambda_n\}_n\subset(1,+\infty)$ such that $\lambda_n\to1$, as $n\to+\infty$, and choose $u_{0,n}=\ff_{\lambda_n}$. Thus, $E(u_{0,n})<E(\ff)$ and
	\[
	\|u_{0,n}\|_{\dot{H}}>\|\ff\|_{\dot{H}}.
	\]
	Recalling Theorem \ref{groun-thm}, if we set $\ff(x,y)=\al_c^{\frac{1}{p-2}}\psi(\al_c^{\frac1{2}}x,\al_c^{\frac1{2s}}y)$, we observe that $\psi$ satisfies \eqref{equ-0} with $\al=1$. A straightforward computation  shows that
	\[
	\begin{split}
	% 	&E(u_{0,n})\|u_{0,n}\|_\lt^{2\rho}<
	% E(\psi)\|\psi\|_\lt^{2\rho},\\
	% &
	% \|u_{0,n}\|_{\dot{H}}\|u_{0,n}\|_\lt^\rho>
	% \|\psi\|_{\dot{H}}\|\psi\|_\lt^\rho.
&\frac{E(u_{0,n})}{E (\psi)}<\paar{\frac{s^s }{(1+s)^{1+s}}}^{\frac{p-2}{(p-2)(1+s)-4s}}
    \|\psi\|_\lt^{2\rho},\\
			 &
 	 		 \frac{\|u_{0,n}\|_{\dot{H}}^2}{\|\psi\|_{\dot{H}}^2}  >\paar{ \frac{s^s }{(1+s)^{1+s}}}^{\frac{p-2}{(p-2)(1+s)-4s}}
	 \|\psi\|_\lt^{2\rho},
	\end{split}
	\]
We note that by using the parabolic regularization theory, we can show the local well-posedness of \eqref{eq0} in $H^{2^+}(\rt)$, so that the corresponding local solutions belong to $C([0,T);L^q)$  for some $q>p$. On the other hand, Theorem \ref{decay-regul-lemma} implies that $u_{0,n}\in H^{2^+}(\rt)$.
 Therefore, the proof follows from Theorem \ref{bl-c-thmm}.
\end{proof}
%%%%%%%%%%%%%%%%%%%%%%%%%%%%%%%%%%
\section{Boosted traveling waves}\label{boosted-section}

%\torange{In this section $2<p<\frac{2(1+s)}{1-s}$?}

In this section, we study the existence of properties of boosted standing waves of \eqref{eq0} when $2<p<\frac{2(1+s)}{1-s}$. By a boosted standing wave we mean a solution of \eqref{eq0} of the form 
\[
\Phi(t,x,y)=\ee^{\ii\alpha t}u(x-\nu t,y-\omega t),
\]
where $\omega ,\nu,\alpha\in\rr$. Substituting the above form into \eqref{eq0}, we see that $u$ satisfies
\[
\alpha u+D_y^{2s}u+\ii\omega u_y-u_{xx}+i\nu u_x=|u|^{p-2}u.
\]
By the Galilean boost transform $u\mapsto \ee^{\frac{i\nu x}{2}}u$, 
 
the term $i\nu u_x$ can been gauged away, and the above equation reduces to
\begin{equation}\label{g-half-wave}
	\alpha u+D_y^{2s}u+\ii\omega u_y-u_{xx} =|u|^{p-2}u.
\end{equation} 
Notice that this transform preserves the mass. 

The solutions of \eqref{g-half-wave} are the critical points of 
\[
S(u)=\frac12\int_\rt\paar{\alpha|u|^2+|u_x|^2+|D_y^su|^2+i\omega\overline{u}u_y}\dd x\dd y
-\frac1p\|u\|_{L^p}^p.
\]
Observe that, if $s=1/2$, then $\scal{u,D_y^{2s}u+i\omega u_y}>0$, for any $u\not\equiv0$, provided $|\omega |<1$. Instead, if $s\in(1/2,1)$, then
\[
\scal{u,D_y^{2s}u+\ii\omega u_y}\geq -\omega_0\|u\|_\lt^2,
\]
where
\begin{equation}\label{omega0}
    \omega_0=(2s-1)\abso{\frac{\omega}{2s}}^{\frac{2s}{2s-1}}.
\end{equation}

 \subsection{Existence and decay of boosted travelling waves}
 The existence of nontrivial solutions of \eqref{g-half-wave} can be obtained via the classical minimization problem
 \begin{equation}\label{wein}
\bW_{\omega}=\inf_{u\in H\setminus\{0\}}\fm_{\om}(u)=\inf_{u\in H\setminus\{0\}}\frac{\paar{ \|u\|_{\dot{H}}^2+\ii\omega \scal{\nh_yD_y^{\frac12}u, D_y^{\frac12}u}+\alpha\|u\|_\lt^2}^\frac p2}{\|u\|_{L^p}^p}.
 \end{equation}
 It is clear, from Lemma \ref{gn-lemma}, that $ \bW_{\omega}>0$. Moreover, if $u$ is nontrivial minimizer of \eqref{wein}, then a scaled function of $u$ satisfies \eqref{g-half-wave}.

 \begin{thm}\label{exis-thm-bt}
 Let $s\in[1/2,1)$ and $2<p<\frac{2(1+s)}{1-s}$.	If $\alpha>\omega_0$, and $|\omega|<1$ when $s=1/2$, then there exists a minimizer $\ff\in H\setminus\{0\}$ of \eqref{wein} which satisfies (after a scaling) \eqref{g-half-wave}. Moreover, if $p$ is an even integer, there exists a  minimizer $u$ of \eqref{wein} which is symmetric respect to $x$. Moreover,
the maps   $y\mapsto\Re u(x,y)$ and $y\mapsto\Im u(x,y)$ are even and odd, respectively, for any fixed $x\in\rr$.
 \end{thm}
To prove Theorem \ref{exis-thm-bt}, we will use the following lemma,  which is a generalization of the Lieb  compactness lemma \cite{Lieb}.  The proof follows the approach to that of \cite[Lemma 2.1]{bfv} and so we only sketch it. 
\begin{lem}\label{lieb}
    Let $0<s<1$ and $\{u_j\}_j\subset \dot H\cap L^q$ with $1<q<+\infty$ such that
\begin{equation}\label{norme}
      \sup_j\paar{\norm{u_j}_{\dot H}+\norm{u_j}_{L^q}}<+\infty  
\end{equation}
and
\begin{equation}\label{misura}
\inf_j\paar{{\rm measure}\sett{(x,y)\in\rt,\;|u_j(x,y)|>\ft}}\geq\fc,
\end{equation}
for some $\ft,\fc>0$.
Then there exists a sequence $\{(x_j,y_j)\}_j\subset\rt$ such that $\{u_j(x+x_j,y+y_j)\}_j$ possesses a subsequence converging weakly in $\dot H\cap L^q$
   to a  function $u\not\equiv0$, as $j\to +\infty$.
\end{lem}

\begin{proof}
Let $\chi_0$ and $\chi_1$ be smooth, real-valued functions on $[0,+\infty)$ such that $\chi_0$ has compact support and it is equal to $1$ near the origin. We decompose $u_j=\chi_0(D)u_j+\chi_1(D)u_j$, where $\chi_0(D)$ and $\chi_1(D)$ denote the Fourier multipliers. By \eqref{misura}, we may assume that either
\begin{equation*}
    \inf_j\paar{{\rm measure}\sett{(x,y)\in\rt,\;|\chi_0(D)u_j(x,y)|>\frac \ft 2}}>0,
\end{equation*}
or 
\begin{equation*}
    \inf_j\paar{{\rm measure}\sett{(x,y)\in\rt,\;|\chi_1(D)u_j(x,y)|>\frac \ft 2}}>0.
\end{equation*}
The former case can be handled as in \cite{bfv}. For the latter one, since $\chi_1$ is supported away from zero, by \eqref{norme} we have that $\sup_j \|\chi_1(|D|)u_j\|_{\dot{H}^\sigma(\R^2)}<\infty$, for $\sigma\in (0,s)$. Here $\dot{H}^\sigma(\R^2)$ denotes the classical fractional homogeneous Sobolev space. So, by means of \cite[Theorem 5.1]{bfv}, we can adapt the arguments of the proof of \cite[Lemma 2.1]{bfv} to conclude.
\end{proof}

 \begin{proof}[Proof of Theorem \ref{exis-thm-bt}.]
The proof is standard, and we give the sketch for the completeness. Let $$\|u\|_Y^2=\|u\|_{\dot{H}}^2+\ii\omega \scal{\nh_yD_y^{\frac12}u, D_y^{\frac12}u}+\alpha\|u\|_\lt^2$$ and $\{u_j\}_j$ be a minimizing sequence for \eqref{wein}. We can assume that $\|u_j\|_{L^p}=1$. Then $\{u_j\}$ is bounded in $H$.    Lemma \ref{embedding}  implies that $\|u_j\|_\lt\lesssim1$, $\|u_j\|_{L^r}\lesssim1$, and $\|u_j\|_{L^{p}} =1$, for any $2<r<\frac{2(s+1)}{1-s}$. By the $pqr$-lemma, 
\[
\inf_j{\rm measure}\sett{(x,y)\in\rt,\;|u_j(x,y)|>\ft}\geq\fc,
\]
for some $\ft,\fc>0$. So, by Lemma \ref{lieb}, there exists $\{(x_j,y_j)\}_j\subset\rt$ such that $\{u_j(x+x_j,y+y_j)\}_j$ has a subsequence that converges weakly in $H$ to a  function $u\in H\setminus\{0\}$, as $j\to +\infty$. Without loss of generality and for the simplicity, we assume that $x_j=y_j=0$, for all $j\ge 1$. It is clear that $\bW_\omega\leq\fm_{\om}(u)$ and $\{u_j\}\subset H$ is a minimizing sequence for $ \bW_\omega$ such that $\|u_j\|_{L^p}=1$. Moreover, up to a subsequence, $v_j\to u$ pointwise for a.e. $(x,y)\in\rt$, as $j\to +\infty$. Now, Lemma \ref{gn-lemma} and the Brezis-Lieb Lemma imply that
\begin{equation*}
\|u_j-u\|_{L^p}^p+\|u\|_{L^p}^p-1=o(1).
 \end{equation*}
Notice that
$
\lim_j \norm{u_j}_Y^2=  \bW_\omega  ^{\frac2p}.
$
Since $u_j$ converges weakly to $u$, as $j\to +\infty$, and $\|\cdot\|_Y\sim\|\cdot\|_H$, then 
\[
 \|u\|_{Y}^2+\|u_j-u\|_{Y}^2-  \bW_\omega ^{\frac2p}=
o(1).
\]
Hence, 
\[
\begin{split}
 \bW_\omega&=\paar{\|u\|_Y^2+\|u_j-u\|_Y^2+o(1)}^{\frac p2}
\geq
\|u_j-u\|_Y^p+\|u\|_Y^p+o(1)\\&
\geq
\bW_\omega\|u_j-u\|_{L^p}^p+\|u\|_Y^p+o(1),
\end{split}
\]
so that, as $j\to+\infty$, we get  $ \bW_\omega\|u\|_{L^p}^p\geq\|u\|_Y^{p}$. 
Since $u$ is nontrivial, we deduce that
$\bW_\omega\geq\fm_{\om}(u)$ and so $u$ is a minimizer of \eqref{wein}.

The Steiner symmetrization can be used, similar to \cite{epb}, to find a minimizer of \eqref{wein} which is symmetric respect to $x$. Let $u^\sharp$ be the Steiner symmetrization respect to $y$ in the Fourier space, that is $u^\sharp=((\hat{u})^{\ast_1})^\vee$ where $\ast_1$ is the symmetric decreasing rearrangement in $y$.   Let $p$ be an even integer. Similar to the proof of Theorem \ref{decay-regul-lemma}, one can show that $u\in H^{(3,2s+1)}$, so $\hat{u}\in L^1(\rt)$. So, analogous to \cite{blss}, we have $u^\sharp\in L^p$,
\[
\|u\|_{L^p}\leq\norm{u^\sharp}_{L^p},
\]
and  then $\fm_{\om}(u^\sharp)\leq\fm_{\om}(u)$. Hence, $u^\sharp$ is also a minimizer. Moreover, since $|\hat{u}|$ is real and belongs to $L^1$, 
\[
u^\sharp(x,y)=(|\hat{u}|)^\vee(x,y)=\overline{(|\hat{u}|)^\vee(x,-y)}
=\overline{u^\sharp(x,-y)}.
\]
So,  $y\mapsto\Re u(x,y)$ and $y\mapsto\Im u(x,y)$ are even and odd, respectively, for any fixed $x\in\rr$.
 \end{proof}
 
  \begin{thm}\label{decay-thm-re}
 	Let $s\geq1/2$ and $2<p<\frac{2(1+s)}{1-s}$.	If $\alpha>\omega_0$,   and $|\omega|<1$ when $s=1/2$, then any solution $\ff\in H\setminus\{0\}$ of \eqref{g-half-wave} satisfies
 	\begin{equation}\label{decay-btw}
 		|y|^2\ee^{\sqrt{\alpha_0}|x|}\ff\in L^\infty,
 	\end{equation}
 	for  some $\alpha_0=\alpha_0(\alpha,\omega_0)>0$.
 \end{thm}
 \begin{proof}
 	Let 
 	\[
 	G(x,y)=\paar{\frac{1}{\al+\xi^2+|\eta|^{2s}-\omega\eta}}^\vee=\int_0^{+\infty}\ee^{-\al t-\frac{x^2}{4t}}t^{-\frac12}\int_\rr\ee^{\ii\eta y} \ee^{-t|\eta|^{2s}+t\omega\eta}  \dd \eta\dd t.
 	\]
  Choose $\omega_1\in(0,\al)$ and $\omega_2=\omega_2(\omega_1)>0$ such that 
  \begin{equation}\label{low-et}
  	  |\eta|^{2s}-\omega\eta+\omega_1\geq\omega_2|\eta|^{2s},
  \end{equation}
  for all $\eta\in\rr$.
  Hence, we have
  \[
\begin{split}
	(y-it\omega)G(x,y) & =
  \ii\int_0^{+\infty}\ee^{-\frac{x^2}{4t}-t\al}t^{-\frac12}
  \int_\rr\ee^{\ii\eta y}\paar{\partial_\eta-t\omega}\ee^{-t(|\eta|^{2s}-\omega\eta)}\dd\eta\dd t\\
  &=
   2\ii s\int_0^{+\infty}\ee^{-\frac{x^2}{4t}-t\al}t^{ \frac12}
  \int_\rr\ee^{\ii\eta y} |\eta|^{2s-2}\eta\ee^{-t(|\eta|^{2s}-\omega\eta)}\dd\eta\dd t.
\end{split}
  \]
   Hence, if we set $\al_0=\al-\omega_1$, we obtain from \eqref{low-et}
%\blue{, and by using the inequality 
%   \begin{equation}\label{dis}
% \frac{x^2}{4t}+\alpha_0 t\ge \frac{\sqrt{\alpha_0}|x|}4+\frac{15\alpha_0 t}{16}      
%   \end{equation}
%  }    
that\
   \begin{equation}\label{upb-1}
  	\begin{split}
  	|yG(x,y)|\lesssim
  	|	(y-\ii t\omega)G(x,y)|
  	&\lesssim
  	   \int_0^{+\infty}\ee^{-\frac{x^2}{4t}-t\al}t^{ \frac12}
  	\int_\rr   |\eta|^{2s-1} \ee^{-t(|\eta|^{2s}-\omega\eta)}\dd\eta\dd t \\&
  	\lesssim
  	\int_0^{+\infty}\ee^{-\frac{x^2}{4t}-t\al_0}t^{ \frac12}
  	\int_\rr   |\eta|^{2s-1} \ee^{-t\omega_2|\eta|^{2s}}\dd\eta\dd t\\&
  	\lesssim
  	\int_0^{+\infty}\ee^{-\frac{x^2}{4t}-t\al_0}t^{ -\frac12}
   \dd t\\&
  =
\sqrt{\frac\pi{\alpha_0}}\ee^{- \sqrt{\al_0} |x|}.
  \end{split}
  \end{equation}
 
    This means that $y\ee^{\sqrt{\al_0}|x|}G\in L^\infty$.\\
  Similar to the above estimates, we get
 
 \[
   \begin{split}
   	|y^2G(x,y)|\lesssim
   	 	\abso{(y-\ii t\omega)^2G(x,y)} 
   	&\lesssim
   	\int_0^{+\infty}\ee^{-\frac{x^2}{4t}-t\al_0}t^{\frac12}\int_\rr |\eta|^{2s-2}\ee^{-t \omega_2|\eta|^{2s}}\dd \eta 
   	\dd t 
    \\& \quad+\int_0^{+\infty}\ee^{-\frac{x^2}{4t}-t\al_0}t^{\frac12}\int_\rr t|\eta|^{4s-2}\ee^{-t \omega_2|\eta|^{2s}}\dd \eta 
   	\dd t .
   \end{split}
 \]  Observe that
 \[
  \int_\rr |\eta|^{2s-2}\ee^{-t \omega_2|\eta|^{2s}}\dd \eta 
 =
 \int_0^1 |\eta|^{2s-2}\ee^{-t \omega_2|\eta|^{2s}}\dd \eta  + \int_1^{+\infty} |\eta|^{2s-2}\ee^{-t \omega_2|\eta|^{2s}}\dd \eta 
 \lesssim
 c_1+c_2t^{-1} 
   \]
   and
   \[
  \int_\rr t|\eta|^{4s-2}\ee^{-t \omega_2|\eta|^{2s}}\dd \eta 
 =
 \int_0^1 t|\eta|^{4s-2}\ee^{-t \omega_2|\eta|^{2s}}\dd \eta  + \int_1^{+\infty} t|\eta|^{4s-2}\ee^{-t \omega_2|\eta|^{2s}}\dd \eta 
 \lesssim
 c_3t+c_4 .
   \]
Hence,   $(1+x^2)y^2\ee^{\sqrt{\al_0}|x|}G\in L^\infty$.
 So,   $y^2\ee^{\sqrt{\al_0}|x|}G\in L^\infty$.
  
 Note that, in the case $s=1/2$, by using
  \[
  \paar{\ee^{-t(|\eta|-\omega\eta)}}^\vee=\frac{t}{t^2+(y-\ii t\omega)^2},
  \]
we can find a better presentation for $G$:
   \[
  \begin{split}
  	 |G(x,y)|&=\abso{\int_0^{+\infty}
   \ee^{-\frac{x^2}{4t}-t\al } 
   \frac{t^{\frac12}}{t^2+(y-\ii t\omega)^2}\dd t}
   \lesssim
   \int_0^{+\infty}
   \ee^{-\frac{x^2}{4t}-t\al } 
   \frac{t^{\frac12}}{y^2+(1-\omega^2)t^2}\dd t\\&
   \lesssim y^{-2}\ee^{-\sqrt{\al_0}|x|}.
  \end{split}
   \] 
   
   Next we show that $G\in L^q$, for all $1\leq q<\frac{1+s}{1-s}$. To do so, first we use the identity
   \[
   G(x,y)=\int_0^{+\infty}\ee^{-\al_0 t}P(x,y,t)\dd t,
   \]
   where
   \[
   P(x,y,t)=\int_\rt\ee^{\ii(x\xi+y\eta )} \ee^{-t\xi^2 -t(|\eta|^{2s}-\omega\eta+\omega_1)}\dd\xi\dd\eta.
   \]
     Let us observe that
   \[
   \abso{P(x,y,t)}\lesssim \int_\rt\ee^{-t\xi^2-t\omega_2|\eta|^{2s}}\dd\xi\dd\eta
   \lesssim t^{-\frac{s+1}{2s}}.
   \]
   On the other hand, similar to \eqref{upb-1}, we derive
   \[
   \abso{(y-\ii t\omega)xP(x,y,t)}\lesssim1.
   \]
   Hence, we get
   \[
   \abso{P(x,y,t)}\lesssim \min\sett{t^{-\frac{s+1}{2s}},|xy|^{-1}}.
   \]
   Consequently, 
   \[
\begin{split}
	   \abso{G(x,y)}&\leq\int_0^{+\infty}\abso{P(x,y,t)}\dd t
   =\int_0^{|xy|^{\frac{2s}{s+1}}}\abso{P(x,y,t)}\dd t
   +
   \int_{|xy|^{\frac{2s}{s+1}}}^{+\infty}\abso{P(x,y,t)}\dd t\\
   &
   \lesssim
   \int_0^{|xy|^{\frac{2s}{s+1}}}|xy|^{-1}\dd t
   +
   \int_{|xy|^{\frac{2s}{s+1}}}^{+\infty} t^{-\frac{s+1}{2s}}\dd t
   %\\&
   \lesssim
   |xy|^{\frac{s-1}{s+1}}.
\end{split}
   \]
   Considering all above estimates, it is revealed that
   \[
 \begin{split}
 	  \|G\|_{L^q}
   &\leq\int_{\{x^2+y^2\leq1\}}|G(x,y)|^q\dd x\dd y
   +
   \int_{\{x^2+y^2\geq1\}}|G(x,y)|^q\dd x\dd y
   \\&
   \lesssim
   \int_{\{x^2+y^2\leq1\}}|xy|^{-\frac{q(1-s)}{1+s}}\dd x\dd y
   +
   \int_{\{x^2+y^2\geq1\}}y^{-2q} \ee^{-\sqrt{\al_0}q|x|}  \dd x\dd y
   <+\infty,
 \end{split}
   \]
      provided $1\leq q<\frac{1+s}{1-s}$. Now, by repeating the proof of Theorem \ref{decay-regul-lemma}, we deduce that $\ff\in L^\infty\cap H^{(3,1+2s)}$. Finally, since $xG,yG\in L^2$, an application of Lemma 3.1 in \cite{chenbona} gives \eqref{decay-btw}. 
 \end{proof}

\subsection{The case $s=1/2$}

In this section we consider $s=1/2$. This case is more delicate because of the competing effects of $\ii\partial_y$ and $D_y^1$. In particular we will prove a non-existence result for $|\omega|>1$ and the absence of small data scattering within the energy space $H$.

\subsubsection{A non-existence result}
Following the ideas of \cite{bglv}, we get the following Pohozaev identity for $s=1/2$.
\begin{lem}\label{nonex-poho}
	Let  $\omega\in\rr$, $s=1/2$, and $\ff\in H$ be a solution of 
	\begin{equation}\label{half-wave}
		\alpha u+ D_y^1 u+\ii\omega  u_y-u_{xx} =f(|u|)u, 
	\end{equation} 
 	where $f:[0,+\infty)$ is a real-valued function, then
	\begin{equation}\label{poho-n}
		\int_{\rr^2}\frac{\eta}{|\eta|} \abso{\widehat{\ff}(\xi,\eta)}^2 \dd\xi\dd\eta
		=
		\omega \|\ff\|_{L^2}^2.
	\end{equation}
\end{lem}
\begin{proof}
Let $L=  D_y^1 +\ii \omega \partial_y$.	For any $a\in\rr$, define $U_au(x,y)=\ee^{\ii a y}u(x,y)$.
	
	Setting $\ff_a=U_a\ff$    and   $L_a =U_aL U_{-a}$, we observe that $\ff_a $ satisfies 
	\begin{equation*} 
	 L_a\ff_a-\partial_{xx}\ff_a+ \alpha  \ff_a=f(|\ff|)\ff_a.
	\end{equation*}
	Moreover, $L_a=M_a+a\omega {\textbf{1}}+\ii \omega \partial_y$ and
	\[
	\scal{\ff,(L_a-L)\ff_a}_{L^2}=0,
	\]
	for all $a\in\rr$, where $M_a$ is the Fourier multiplier of $|\eta -a|$.    It is not hard to see, as $a\to0$, that 
	$
	\frac{ L_a-L }a
	$
	converges weakly   in ${L^2}$ to $-\ii\mathcal{H}_y +\omega$, where $\mathcal{H}_y$ is the Hilbert transform on the $y$-variable.
	Since $\ff_a$ converges strongly to $\ff$ in ${L^2}$, as $a\to0$, then
	\[
	\scal{\ff,\ii\mathcal{H}_y\ff}=\omega  \|\ff\|_{L^2}^2
	\]
 and so \eqref{poho-n} follows.
 \end{proof}
 \begin{thm}\label{nonexi}
 	There is no nontrivial solution of \eqref{half-wave} if $|\omega |>1$.
 \end{thm}
 \begin{proof}
 	It follows, from Lemma \ref{nonex-poho}, that
 	\[
 	|\omega |\|\ff\|_{L^2}^2=	\abso{\scal{\ff,\mathcal{H}_y\ff}}
 	\leq 
 	\|\ff\|_{L^2}
 		\|\mathcal{H}_y\ff\|_{L^2}
 		\leq 
 		\|\ff\|_{L^2}^2.
 	 	\]
Hence, $\ff\equiv0$, if $|\omega |>1$.
 \end{proof}

\subsubsection{No small data scattering}
Our aim, now, is to show the absence of small data scattering for \eqref{eq0} within the energy space $H$, in the case  $s=1/2$ with $2<p<6$. 
\\
Define
\[
W_\omega=
\inf_{u\in H\setminus\{0\}}
\frac{\paar{\nl_\omega(u)}^{\frac{3p-6}{4}}\|u\|_\lt^{\frac{6-p}{2}}}
{\|u\|_{L^p}^p},
\]
where  
\begin{equation}\label{lowwer}
\nl_\omega(u)=\|u\|_{\dot{H}}^2+
\ii\omega \scal{\nh_yD_y^{\frac12}u, D_y^{\frac12}u}
\geq (1-|\omega|)\|u\|_{\dot{H}}^2. 
\end{equation}
  Adapting the proof of \cite[Appendix B]{fjl} and \cite{epb},  and employing the concentration-compactness principle, one can show that every
minimizing sequence for $W_\omega(u)$ is relatively compact in $H$ up to translations and scalings.

\begin{lem}
There holds
\begin{equation}\label{est-min}
	W_\omega\sim
(1-|\omega|)^{\frac{3(p-2)}{4 }}.
\end{equation}
\end{lem}

\begin{proof}
We have, from \eqref{gn}, that
\[
\|u\|_{L^p}^p
\leq C_{p}
\|u\|_\lt^{\frac{6-p}{2}}
\|u\|_{\dot{H}}^{\frac{3p-6}{2}}
\leq
C_{p}(1-|\omega|)^\frac{6-3p}{ 4}
\|u\|_\lt^{\frac{6-p}{2}}
(\nl_\omega(u))^\frac{3p-6}{4}.
\]
Thus,
\[
W_\omega\geq C_p^{-1}(1-|\omega|)^{\frac{3p-6}{4 }}.
\]
Without loss of generality, we can assume that $\om>0$. Let us fix $g\in H^{1/2}(\rr) \setminus\{0\}$ such that $\text{supp}(\hat{g})\subset[0,+\infty)$. For $\epsilon>0$, let $\phi_\epsilon:\rr \to \rr$ be so defined:
%\torange{since $\eps \to 0$, in order to have an approximate identity we should have $c_\epsilon\ee^{-\frac{\xi^2}{\eps}}$}
\[
\phi_\epsilon(\xi)=c_\epsilon\ee^{-\frac{\xi^2}{\eps}}, \qquad \text{for }\xi \in \rr,
\]
where $c_\epsilon>0$ is such that $\|\phi_\epsilon\|_{L^2(\rr)}=1$.
Finally, let $u_\epsilon\in H$ be such that 
\[
\hat{u}_\eps(\xi, \eta)=\phi_\epsilon(\xi)\hat{g}(\eta), \qquad \text{for }(\xi,\eta) \in \rr^2.
\]
% \torange{$\{\phi_\epsilon^2\}$ should be an approximate identity and not $\{\phi_\epsilon\}$}
Since $\{\phi_\epsilon^2\}$ forms an approximate identity on $\rr$, we have
\[
\|u_\eps\|_{\dot{H}}^2
=\int_{\rr^2}(\xi^2+|\eta|)|\phi_\eps(\xi)|^2|\hat{g}(\eta)|^2\dd \xi \dd \eta \to 
\int_{\rr}|\eta||\hat{g}(\eta)|^2\dd \eta= \|g\|_{\dot{H}^{1/2}(\rr)}^2,
\qquad \text{as }\epsilon \to 0.
\]
Moreover, since $\text{supp}(\hat{g})\subset[0,+\infty)$ and observing that $\xi^2+\eta-\om \eta\le (1-\om)(\xi^2+\eta)+ \xi^2$, for all $(\xi,\eta)\in\rr\times[0,+\infty)$, we deduce that
\[
\begin{split}
    \nl_\omega(u_\epsilon)
&=\int_{\rr\times[0,+\infty)}(\xi^2+\eta-\om \eta)|\phi_\eps(\xi)|^2|\hat{g}(\eta)|^2\dd \xi \dd \eta
\\
&\le (1-\om)\int_{\rr\times[0,+\infty)}(\xi^2+\eta)|\hat{u}_\eps(\xi,\eta)|^2\dd \xi \dd \eta
+\|g\|_{L^2(\rr)}^2
\int_{\rr}\xi^2|\phi_\eps(\xi)|^2\dd \xi 
\\
&=(1-\omega)\|u_\eps\|_{\dot{H}}^2+o(1),
\end{split}
\]
where $o(1)\to 0$, as $\eps\to 0$. Since $\|u_\eps\|_{\dot{H}}^2
\to \|g\|_{\dot{H}^{1/2}(\rr)}^2\neq 0$, as $\eps\to 0$, for $\eps_0$ sufficiently small we can conclude that
\[
\nl_\omega(u_{\epsilon_0})
\le 2(1-\omega)\|u_{\eps_0}\|_{\dot{H}}^2.
\]
% \red{On the other hand, if we take $g(x,y)\in H$ such that $\text{supp}(\hat{g})\subset\rr\times[0,\infty)$. Then
% \[
% \nl_\omega(g)\leq(1-|\omega|)\|g\|_{\dot{H}}^2.
% \]}
Hence, we get
\[
W_\omega\lesssim 
(1-|\omega|)^\frac{3p-6}{4 }.
\]
Therefore, we obtain \eqref{est-min}.   
\end{proof}

\begin{lem}\label{scat-lem}
	Let $u_\omega$ be a solution of \eqref{half-wave} with $|\omega|<1$ and $\alpha\geq0$. If there exists $\phi\in H$ such that
	\[
	\norm{\ee^{\ii t\alpha}u_\omega(x,y-\omega t)-\ee^{-\ii t(D_y^1-\partial_{xx})}\phi}_H\to0,
	\]
	as $t\to+\infty$ (or $t\to-\infty$), then $\phi\equiv u_\omega\equiv0$.
\end{lem}
\begin{proof}
	Let $f^\pm$ be    the projections of $f\in H$ onto the positive and negative Fourier frequencies in $y$-variable, respectively, whence we consider the orthogonal decomposition  $H=H^-\oplus H^+$. Hence,  by the assumption we have
	\[
	\norm{\ee^{\ii t\alpha}u^\pm_\omega(\cdot,\cdot-\omega t)-U(-t)\phi^\pm}_H\to0, \qquad\text{as}\, t\to+\infty,
	\]
where
	\[
	U(t)=\ee^{ \ii t(D_y^1-\partial_{xx})}.
	\]
 	Since $U(t)$ is unitary group on $H^\pm$, by using the fact $\ee^{\ii tD_y^1}f(\cdot)=f(\cdot+t)$,  we have
	\[
	\norm{\ee^{\ii t\alpha}\ee^{\ii t\partial_{xx}} u^\pm_\omega(\cdot,\cdot-(\omega-1) t)-\phi^\pm}_H\to0,\qquad\text{as}\, t\to+\infty,
		\] 
deducing
		\[
		\ee^{\ii t\alpha}\ee^{\ii t\partial_{xx}}u^\pm_\omega(\cdot,\cdot-(\omega-1) t)\rightharpoonup\phi^\pm\qquad\text{in}\, L^2,
  \qquad\text{as}\, t\to+\infty.
			\]
On the other hand, it is easy to see that 
	\begin{equation}\label{weak-c}
				\ee^{\ii t\alpha}\ee^{\ii t\partial_{xx}}f(\cdot,\cdot-(\omega-1) t)\rightharpoonup0\qquad\text{in}\, L^2,
    \qquad\text{as}\, t\to+\infty,
			\end{equation}
		  for all $f\in L^2$. Alternatively, \eqref{weak-c} is derived for $u_\omega $from the dispersive estimate
	\begin{equation}\label{est-1w}
				\norm{\ee^{\ii t\partial_{xx}}f}_{L^q(\rr)}\lesssim |t|^{-\frac12\paar{1-\frac2q}}
				\|f\|_{L^{\frac{q}{q-1}}(\rr)},\qquad1\leq q\leq 2.
			\end{equation} 
			More precisely, since 
			$	\ee^{\ii t\alpha} \ee^{\ii t\partial_{xx}}u_\omega(\cdot,\cdot-(\omega-1) t)$ is bounded by $\|u_\omega\|_{L^2}$, we obtain from Theorem \ref{decay-thm-re} that
			\[
			\abso{\scal{	\ee^{\ii t\alpha}\ee^{\ii t\partial_{xx}}u_\omega(\cdot,\cdot-(\omega-1) t),\psi}}
			\leq|t|^{-\frac12\paar{1-\frac2q}}
			\|u_\omega\|_{L^{\frac{q}{q-1}}}\|\psi\|_{L^{\frac{q}{q-1}}}\to0,\qquad\text{as}\, t\to+\infty,
			\]
			for some $q\in(1,2)$ and all $\psi\in C_0^\infty(\rt)$. Thus   \eqref{est-1w} for $f=u_\omega$ follows from the density. Therefore $\phi=\phi^++\phi^-\equiv0$ and so $\|\ee^{\ii t\alpha}u_\omega(\cdot,\cdot-\omega t)\|_H\to 0$, as $t\to +\infty$. Since $\|\ee^{\ii t\alpha}u_\omega(\cdot,\cdot-\omega t)\|_H=\|u_\omega\|_H$, for any $t$, we conclude that $u_\omega\equiv0$. 
   \\
   The argument works also as $t\to-\infty$. And the proof is complete.
		\end{proof}

Now we can prove the  absence of small data scattering for \eqref{eq0} within the energy space $H$.
\begin{thm}\label{nonscat}
	For any $\epsilon>0$, there exists a global solution $u\in C([0,\infty);H)$ of  \eqref{eq0} such that $0<\|u(t)\|_H<\epsilon$, for any $t>0$, and such that there is no $\phi\in H$ satisfying
	\begin{equation*}
 \norm{u(t)- \ee^{-\ii t(D_y^1-\partial_{xx})}\phi}_H\to0,
	\end{equation*}
	as $t\to+\infty$ (or $t\to-\infty$).
\end{thm}

\begin{proof}
	Let $u_\omega$ be given from Theorem \ref{exis-thm-bt}. Since $S'(u_\omega)=0$, by using the identities  (see the proof of Theorem 2.1 in \cite{epb}),
 and get the 
   \[
	\frac{2(p-2)}{6-p}\|u_\omega\|_\lt^2=	
  \norm{D_y^{\frac12} u_\omega } _\lt^2+\ii\omega\scal{u_{\om},(u_{\om})_y} 
	=2\|(u_\omega)_x\|_\lt^2
	=\frac{p-2}{p}\|u_\om\|_{L^p}^p,
		\] 
it is readily seen	 that 
\[
(1-|\omega|)^{\frac{3(p-2)}{4}}\sim W_\omega\simeq\norm{u_\omega}_\lt^{p-2}=
\paar{  \norm{D_y^{\frac12} u_\omega } _\lt^2+\ii\omega\scal{u_\om,(u_\om)_y}+
  \|(u_\omega)_x\|_\lt^2}^{\frac{p-2}{2}}.
\]
Hence, we get from \eqref{lowwer} that
\[
\norm{u_\omega}_\lt\sim(1-|\omega|)^{\frac3{ 4}},\qquad
\norm{u_\omega}_{\dot{H}}\lesssim  (1-|\om|)^\frac{1}{4}.
\]	
Hence, we can find that
\[
\norm{u_\omega}_H\lesssim  (1-|\om|)^\frac{1}{4}+(1-|\omega|)^{\frac3{4}}<\epsilon
\]
provided $|\omega|<1$ is sufficiently close to $1$. Hence, $u_\omega$ satisfies Lemma \ref{scat-lem}. This completes the proof. 
\end{proof}

\subsection{Normalized solutions in the mass subcritical case}
 
 Another interest in the study of \eqref{g-half-wave} is the $L^2$-constrained solutions of \eqref{g-half-wave} in the mass subcritical case. Indeed, we are interested in finding the minimizers of
 \	\begin{align} \label{g-min}
 m_\omega(c):=\inf_{u \in S_c} E_\omega(u),
 \end{align}
 where
 \[
 E_\omega(u)=E(u)+\frac{\ii\omega}{2}\scal{\nh_yD_y^{\frac12}u, D_y^{\frac12}u}.
 \] 
Notice that 
\[
m_0(c)=m(c),
\]
 where $m(c)$ is the same as in \eqref{min}. It is clear that $m_\omega(c)=\inf_{v\in S(1)} E_\omega(\sqrt{c}v)$. Since $E_\omega(\sqrt{c}v)$ is concave in $c\in(0,\infty)$, so $m_\omega(c)$ is.
 
We claim that $m_\omega (c)<0$, for all $c>0$. Let us start with the case $s>1/2$.  Let $\epsilon>0$, $\chi\in C_0^\infty(\rt)$ and $K>0$ such that $\text{supp}(\chi)\subset[0,\epsilon]\times[\eta_0-\epsilon,\eta_0+\epsilon]$ such that $u=K(\chi)^\vee\in S_c$, where $\eta_0$ is the unique root of $\eta^{2s}-\omega\eta+\omega_0$ (for the definition of $\omega_0$ see \eqref{omega0}). It is clear that $u\in\mathcal{S}(\rt)$. Hence, we get 
 \[
\begin{split}
	m_\omega(c)+\frac{\omega_0}{2}c
&\leq E_\omega (u)+
\frac{\omega_0}{2}\|u\|_\lt^2
\leq \frac12\int_\rt\paar{\xi^2+\eta^{2s}-\omega\eta+\omega_0}|\hat{u}(\xi,\eta)|^2\dd\eta\dd\xi\\
&\lesssim O( \epsilon^2 )\|u\|_\lt^2=O( \epsilon^2 )c.
\end{split}
 \]
 This shows that $$m_\omega(c)\leq -\frac{\omega_0c}{2}<0.$$ On the other hand, when $s=1/2$, since $E_\omega(u_t)<0$ as $t\to0^+$, then
 $m_\omega(c)<0$ for all $c>0$.

 By using the inequality
 \[
\scal{i\mathcal{H}_yD_y^{\frac12}u,D_y^\frac12u}\leq
 \norm{D_y^\frac12 u} _\lt^2\leq C_s
 \|u\|_\lt^{\frac{2s-1}{s}}\norm{D_y^su}_\lt^\frac{1}{s},
 \]
 we see, for any $u\in S_c$, that
\begin{equation*}%\label{lower-bound-BTW}
	\begin{split}
	 E_\omega(u)&\geq \frac12\|u\|_{\dot{H}}^2-\frac{|\omega|C_s c^{\frac{2s-1}{s}}}{2}  
\norm{D_y^su}_\lt^\frac{1}{s}
-\frac{C_{p,s} }{p}c^{\frac p 2-\frac{(p-2)(1+s)}{4s}}
\norm{u_x}^{\frac{p-2}{2}}_\lt
\norm{D_y^{s} u}^{\frac{p-2}{2s}}_\lt\\
&
 \geq 
\frac12\|u\|_{\dot{H}}^2-\frac{|\omega|C_s c^{\frac{2s-1}{s}}}{2}  
\norm{u}_{\dot{H}}^\frac{1}{s}
-
\frac{C_{p,s} }{p}c^{\frac p 2-\frac{(p-2)(1+s)}{4s}}
\|u\|_{\dot{H}}^{\frac{(p-2)(s+1)}{2s}}.
\end{split}
\end{equation*}
 Notice that  if $s=1/2$, then $$\scal{u,D_y^{2s}u+i\omega u_y}\sim\|D_y^\frac12u\|_{L^2}^2$$ provided $|\omega |<1$.
  
 Hence, $m(c)$ is bounded from below, if $2<p<\frac{2(3s+1)}{s+1}$. Then the argument of Theorem \ref{thm1}, with some modifications, guarantees the existence of minimizers of \eqref{g-min} in the subcritical case  $2<p<\frac{2(3s+1)}{s+1}$.

	\begin{thm}\label{boost-norm}   
	Let $2<p<\frac{2(3s+1)}{s+1}$, then any minimizing sequence to \eqref{g-min} is compact in $H$, up to translations, for any $c>0$. In particular, there exist minimizers to \eqref{g-min}, for any $c>0$.
\end{thm}

 %%%%%%%%%%%%%%%%%%%%%%%%%%%%%%%%%%%%%%%%%%%%%%%%%%%%%%%%
  
%%%%%%%%%%%%%%%%%%%%%%%%%%%%%%%%%%%%%
% 	\section*{Appendix}
	
% Here, we revisit  the existence of solutions to \eqref{equ-0}-\eqref{mass} in the mass critical case  for any $c>c_*$, where $c_*>0$ is the constant given in Theorem \ref{thm1-mass-cr}.
	
% 	 Define
% 	$$
% 	\mathcal{S}(c):=\left\{ u \in S_c :\int_{\R^2} |\partial_x u|^2  \dd x\dd y + \int_{\R^2} |D_y^{s} u|^2  \dd x\dd y<\frac{3s+1}{s+1}\int_{\R^2}|u|^p  \dd x\dd y\right\}.
% 	$$
	
% 	Replacing the roles of $S_c$ in the mass supercritical case by $\mathcal{S}(c)$, we are now able to obtain the following existence result.
	
% 	\begin{thm} \label{thm5}
% 		Let $p=\frac{2(3s+1)}{s+1}$, then there exists a ground state solution to \eqref{equ-0}-\eqref{mass} at the level $\mathcal{\gamma}(c)$ for any $c>c_*$, where
% 		$$
% 		\mathcal{\gamma}(c):=\inf_{u \in \mathcal{P}(c)} E(u),
% 		$$
% 		and $\mathcal{P}(c):=\{u \in \mathcal{S}(c) : Q(u)=0\}$.
% 	\end{thm}
% 	\begin{proof}
% 		The proof of this theorem can be done by closely following the ways of the proof of Theorem \ref{thm4}, then we omit it here.
% 	\end{proof}

%%%%%%%%%%%%%%%%%%%%%%%%%%%%%%%%%%%%%%%%%%%%%%%%%
\section*{Conflict of interest} The authors declare that they have no conflict of interest. 

\section*{Data Availability}
There is no data in this paper.
	\section*{Acknowledgment}
	A. E. is supported by Nazarbayev University under the Faculty Development Competitive Research Grants Program for 2023-2025 (grant number 20122022FD4121).

 A. P. is partly financed by European Union - Next Generation EU - PRIN 2022 PNRR ``P2022YFAJH Linear and Nonlinear PDE's: New directions and Applications" and by INdAM - GNAMPA Project 2023 ``Metodi variazionali per
alcune equazioni di tipo Choquard".
	
	%%%%%%%%%%%%%%%%%%%%%%%%%%%%%%%%%%%%%%%%%%%%%	

\end{document}